\documentclass{amsart}
\usepackage{amssymb,amsmath,amsthm,amsfonts,latexsym,graphicx,hyperref,tikz,microtype}

\let\eps=\varepsilon

\mathchardef\ordinarycolon\mathcode`\:
\mathcode`\:=\string"8000
\begingroup \catcode`\:=\active
  \gdef:{\mathrel{\mathop\ordinarycolon}}
\endgroup

\let\cal=\mathcal
\let\bb=\mathbb

\newcommand{\F}{{\bb{F}}} 
\newcommand{\G}{{\bb{G}}} 
\newcommand{\N}{{\bb{N}}} 
\newcommand{\Q}{{\bb{Q}}} 
\newcommand{\R}{{\bb{R}}} 

\newcommand{\angs}[1]{{\left\langle #1\right\rangle}}
\newcommand{\abs}[1]{\left|#1\right|}
\newcommand{\ceil}[1]{{\left\lceil #1\right\rceil}}
\newcommand{\floor}[1]{{\left\lfloor #1\right\rfloor}}

\newtheorem*{theorem*}{Theorem}
\newtheorem{thm}{Theorem}[section]
\newtheorem*{thm*}{Theorem}
\newtheorem{lem}[thm]{Lemma}
\newtheorem*{lem*}{Lemma}
\newtheorem{cor}[thm]{Corollary}
\newtheorem*{cor*}{Corollary}
\newtheorem{conj}[thm]{Conjecture}
\newtheorem*{conj*}{Conjecture}
\newtheorem{prop}[thm]{Proposition}
\newtheorem{prop*}{Proposition}

\theoremstyle{definition}
\newtheorem{rmk}[thm]{Remark}
\newtheorem*{rmk*}{Remark}

\newtheorem*{obs*}{Observation}

\newtheorem*{prb*}{Problem}

\newtheorem*{q*}{Question}
\newtheorem{exmp}[thm]{Example}
\newtheorem*{exmp*}{Example}

\newtheorem*{exer*}{Exercise}
\newtheorem{defn}[thm]{Definition}
\newtheorem*{defn*}{Definition}

\DeclareMathOperator{\GL}{GL}

\DeclareMathOperator{\im}{Im}
\DeclareMathOperator{\id}{id}
\DeclareMathOperator{\rank}{rank}

\usetikzlibrary{decorations.pathreplacing}

\newcommand{\U}{\cal U}
\newcommand{\kk}{\mathbf{k}}

\newcommand{\M}{\cal M}
\newcommand{\EE}{\cal E}
\renewcommand{\G}{\mathbf{G}}
\renewcommand{\H}{\mathbf{H}}

\newcommand{\der}[2]{ \frac{d#1}{d#2} }

\DeclareMathOperator{\diag}{diag}
\DeclareMathOperator{\shape}{sh}
\DeclareMathOperator{\n}{\mathbf{n}}
\DeclareMathOperator{\cp}{cp}
\DeclareMathOperator{\comm}{Comm}

\newcommand\scalemath[2]{\scalebox{#1}{\mbox{\ensuremath{\displaystyle #2}}}}

\let\mid:

\title[Upper bounds on unitriangular conjugacy classes]{Upper bounds on the number of conjugacy classes in unitriangular groups}
\date{\today}

\author{Andrew Soffer}
\thanks{Department of Mathematics, UCLA, Los Angeles, CA 90095.\\
    \indent Email: \texttt{asoffer@math.ucla.edu}}

\numberwithin{equation}{section}

\begin{document}
\maketitle

\begin{abstract}
    We provide a new upper bound on the number of conjugacy classes in the group $U_n(q)$ of unitriangular matrices over a finite field.
    We also compute a similar upper bound for every group in the lower central series of~$U_n(q)$.
\end{abstract}

\section{Introduction}
For a prime power $q$, let $\F_q$ denote the field with $q$ elements.
Let $U_n(q)$ denote the group of upper-triangular matrices over $\F_q$ with 1s on the diagonal (called the \emph{unitriangular group}).
Let $\kk(G)$ denote the number of conjugacy classes in a finite group $G$.
In 1960, Higman proved the following \cite{higman1960enum1}:

\begin{thm}[Higman]\label{thm:higman_bound}
    For every prime power $q$, \[q^{\frac{n^2}{12}(1+o(1))}\le \kk(U_n(q))\le q^{\frac{n^2}{4}(1+o(1))},\]
    where $o(1)$ means a function of $n$, independent of $q$, which tends to zero as $n$ tends to infinity.
\end{thm}

Higman's original interest was in enumerating finite $p$-groups of a given order.
He obtained an upper bound for the number of groups of order $p^n$ in terms of~$\kk(U_n(p))$.
While the asymptotics of the number of $p$-groups has since been resolved via different methods (see \cite{blackburn2007enumeration}), the gap between the lower and upper bounds for~$\kk(U_n(q))$ has not been closed.

In their 1992 paper, Arregi and Vera-L\'opez used their technique of canonical matrices to improve Higman's upper bound \cite[Theorem~5.4]{vla1992conjugacy}.
They show \[\kk(U_n(q))\le(n-1)!2^{n-1} q^{\frac{n^2+n}{6}}.\]
Note that $(n-1)!2^{n-1}=q^{O(n\log n)}$, and so these terms do not contribute significantly to the asymptotics of $\kk(U_n(q))$.
Using the theory of supercharacters (see \cite{diaconis2008supercharacters, marberg2011combinatorial}), Marberg obtained an upper bound with the same asymptotics \cite[Theorem~5.1]{marberg2008constructing}.

We improve on these asymptotics, with the following result:
\begin{thm}\label{thm:conj_bound}
    For every positive integer $n$ and every prime power $q$, we have \[\kk(U_n(q))\le p(n)^2n!\ q^{\frac{7}{44}n^2+\frac{n}{2}},\]
    where $p(n)$ denotes the number of integer partitions of $n$.
\end{thm}

Our approach is to estimate the number of pairs $(A,B)$ of commuting matrices in~$U_n(q)$ by conjugating $A$ into Jordan canonical form, and determining the possibilities for the image of $B$ under this conjugation.
There are many choices for matrices which conjugate $A$ into Jordan form, and the image of $B$ depends on this choice.
Section~\ref{sec:jordan_forms_and_conj} defines our canonical choice $X_A$ for conjugation.

For each upper-triangular matrix $A$ we conjugate its centralizer $C_\U(A)$ via our canonical choice $X_A$ defined in Section~\ref{sec:jordan_forms_and_conj}.
The resulting space $X_AC_\U(A)X_A^{-1}$ can often be described by a combinatorial object which we call a \emph{gap array}.
Section~\ref{sec:gap} introduces gap arrays and proves several structural lemmas about them.
While not every space $X_AC_\U(A)X_A^{-1}$ can be encoded by a gap array, every such space is a subspace of one encoded by a gap array.
Determining the sizes of these subspaces via the combinatorics of gap arrays, we obtain the same upper bound as Marberg and Arregi and Vera-L\'opez (see Corollary \ref{cor:sixth}).
However, the technique of gap arrays is amenable to further optimization.
These optimizations are the content of the proof of Theorem \ref{thm:better_conj_bound}, from which Theorem \ref{thm:conj_bound} follows (see Section \ref{sec:proof}).

In Section \ref{sec:lower_central}, we develop another set of tools for computing bounds on the number of conjugacy classes in groups in the lower central series for $U_n(q)$.
We prove:
\begin{thm}\label{thm:lower_central_bound}
    Let $U_{n,k}(q)$ denote the $k$th group in the lower central series for $U_n(q)$, and let
    \[\gamma_m:=\frac{1}{6}-\frac{13}{24}\cdot 4^{-m}+2^{-(m+1)}-4^{-(m+1)}m.\]
    Then for every $q$,
    \[\kk(U_{n,k}(q))\le q^{\gamma_m n^2(1+o_m(1))},\]
    where $m=\floor{\log_2\left(\tfrac{n}{k}\right)}$, and $o_m(1)$ denotes a function which, for each fixed $m$, tends to zero as $n$ tends to infinity.
\end{thm}
The techniques we use to establish this result are entirely unrelated to gap arrays.
The main idea for the proof is to bound the number of pairs of commuting matrices in $\U_{n,k}(q)$ by splitting each of the matrices into three parts (see Figure \ref{fig:decomposition}).
Two of these parts are upper-triangular, and form smaller instances of the same problem.
We use the fact that pairs of these smaller upper-triangular parts commute to bound the number of viable completions of the full upper-triangular matrices in terms of the ranks of the smaller parts.
This bound is the content of Lemma~\ref{lem:key_lower_central}.

We then bound the number of pairs of commuting matrices in $\U_{n,k}(q)$ by conditioning on the ranks of the smaller triangular parts.
On the one hand, if the smaller triangular parts must have small rank, then the number of such matrices is a limiting factor.
On the other hand, if the ranks of the triangular parts are large, there are stringent conditions on the viable completions, limiting the total number of possibilities.

\section{Basic Definitions and Notation}\label{sec:notation}
\subsection{Algebra}
For a prime power $q$, let $\F_q$ denote the field with $q$ elements, and let $\U_n(q)$ denote the $\F_q$-algebra of strictly upper-triangular matrices.
Define the \emph{unitriangular group} to be $U_n(q)=\{1+X\mid X\in\U_n(q)\}$.
The field $\F_q$ will almost always be clear from context.
In such cases we will omit the parameter $q$.

For an associative algebra $\cal A$, define \[\comm(\cal A)=\{(X,Y)\in\cal A\times A\mid XY=YX\}.\]
For a group $G$, we use the same notation $\comm(G)$ to denote the set of pairs of commuting elements in $G$.
Note that $X,Y\in\U_n$ commute if and only if the elements $1+X,1+Y\in U_n$ commute, so $\abs{\comm(\U_n)}=\abs{\comm{U_n}}$.

From Burnside's lemma, the number of conjugacy classes $\kk(G)$ in a finite group~$G$ can be computed by
\[\kk(G)=\frac{1}{\abs{G}}\sum_{g\in G}\abs{C_G(g)}=\frac{\abs{\comm(G)}}{\abs{G}},\] where $C_G(g)$ denotes the centralizer of $g\in G$.

For positive integers $a$ and $b$, let $\M_{a\times b}$ denote the vector space of ($a\times b$)-matrices over $\F_q$.
We write $C_{\M}(X)$ for the centralizer of $X$ in $\M_{n\times n}$. That is,
\[C_{\M}(X):=\{A\in\M_{n\times n}\mid AX=XA\}.\]
Similarly, define $C_{\U}(X):=\{A\in\U_n\mid AX=XA\}$.

\subsection{Partitions}
Recall standard notation in partition theory (see, e.g.~\cite{macdonald1995symmetric,stanley2011enum1}).
Let $p(n)$ denote the number of integer partitions of $n$.
For a partition $\lambda\vdash n$, we write $\lambda'$ for the conjugate partition.
That is, $\lambda'=(\lambda'_1,\lambda'_2,\dots)$, where $\lambda'_i$ is the number of parts in $\lambda$ of size at least $i$.
We use $\ell(\lambda)$ to denote the number of parts (or \emph{length}) of $\lambda$.
Clearly $\ell(\lambda)=\lambda'_1$.

For a partition $\lambda\vdash n$, we use the notation 
\begin{equation}\label{eq:n_fn}
\n(\lambda):=\sum_i (i-1)\lambda_i=\sum_i\binom{\lambda_i'}{2}.
\end{equation}
It is common to use the character ``$n$'' for both this function and the size of a partition.
We choose to use the bold roman ``$\n$'' for the function to avoid confusion.

One can also write a partition as $(1^{m_1}2^{m_2}3^{m_3}\dots)$, where $m_i$ denotes the number of parts of size $i$.
Clearly, $m_i$ can be computed by $m_i=\lambda'_i-\lambda'_{i+1}$.

\subsection{Jordan canonical forms}
Let $\lambda=(\lambda_1,\dots,\lambda_\ell)$ be a partition of $n$.
By $J_\lambda$, we denote the Jordan canonical nilpotent matrix which has blocks of size~$\lambda_1,\lambda_2,\dots,\lambda_\ell$.
For example, \[J_{(3,2)}=\left(\begin{array}{ccc|cc}
        0 & 1 & 0 & 0 & 0\\
        0 & 0 & 1 & 0 & 0\\
        0 & 0 & 0 & 0 & 0\\\hline
        0 & 0 & 0 & 0 & 1\\
        0 & 0 & 0 & 0 & 0
\end{array}\right).\]

Every upper-triangular matrix $A\in\U_n$ is $\GL$-conjugate to some $J_\lambda$.
We say that~$\lambda$ is the \emph{shape} of $A$ if $A$ is $\GL$-conjugate to $J_\lambda$, and write $\shape(A)=\lambda$.
For~$A\in\U_n$, and~$1\le k\le n$, let $A|_k$ denote the top-left $(k\times k)$-submatrix of $A$.

For a partiton $\lambda\vdash n$, let $F_\lambda(q)$ denote the number of matrices in $\U_n(q)$ of Jordan type $\lambda$.
Yip proved that $F_\lambda$ is a polynomial in $q$ with integer coefficients and  degree equal to $\binom{n}{2}-\n(\lambda)$ \cite{yip2013q}.
Moreover, Yip showed that the leading coefficient is~$f^\lambda$, the number of standard Young tableaux of shape~$\lambda$.
It is clear from Yip's proof that 
\begin{equation}\label{eq:yip}
    F_\lambda(q)\le f^\lambda q^{\binom{n}{2}-\n(\lambda)}.
\end{equation}
The only fact we will use regarding standard Young tableaux is that for every $\lambda\vdash n$, we have $f^\lambda\le\sqrt{n!}$ (see, e.g.~\cite{sagan2001symmetric}).


\subsection{Partitions as elements of $\ell^1(\N)$}
It will be useful to treat partitions as infinite decreasing sequences of positive integers $(\lambda_1,\lambda_2,\dots)$ where $\lambda_i=0$ for $i>\ell(\lambda)$.
In this way, we can treat $\lambda$ as an element of $\ell^1(\N)$.
Moreover, as $\ell^1(\N)\subseteq\ell^2(\N)$, it makes sense to talk about the inner product of two partitions.
For partitions $\lambda$ and~$\mu$ (not necessarily of the same integer) define
\begin{align*}
    \angs{\lambda,\mu} &:= \sum_{i}\lambda_i\mu_i,\text{ and}\\
    \|\lambda\| &:= \sqrt{\angs{\lambda,\lambda}}.
\end{align*}
Note that combining \eqref{eq:n_fn} with this analytic notation, the function $\n$ can be expressed as
\begin{equation}\label{eq:n_as_ell_two}
    \n(\lambda)=\frac{\|\lambda'\|^2}{2}-\frac{n}{2}.
\end{equation}

Similarly, because the multiplicity $m_i$ of the part $i$ in $\lambda$ can be computed as $m_i=\lambda'_i-\lambda'_{i+1}$, we have
\begin{equation}\label{eq:m_as_ell_two}
(m_1,m_2,\dots)=v-Lv,
\end{equation}
where $L$ denotes the left-shift operator on $\ell^1(\N)$.

\begin{rmk*}
    We define $\|\cdot\|$ to be the $\ell^2$-norm, rather than the~$\ell^1$-norm.
    Treating~$\lambda\vdash n$ as an element of $\ell^1(\N)$, its $\ell^1$-norm is simply $n$, and therefore does not merit its own notation.
    When we need the $\ell^1$-norm of a vector $v\in\ell^1(\N)$, we will explicitly add a subscript to the norm, as in $\|v\|_1$.
\end{rmk*}

\section{Jordan forms and conjugation}\label{sec:jordan_forms_and_conj}
For an upper-triangular matrix $A\in\U_n$, our approach to understanding the size of its centralizer $C_{\U}(A)$ will be to conjugate $A$ into its Jordan form $J_\lambda$ by some $X_A$.
We then analyze $X_AC_{\U}(A)X_A^{-1}$, noting that it is a subspace of $C_{\M}(J_\lambda)$.
Given $A\in\U_n$, there is more than one choice for a matrix $X_A$ satisfying $X_AAX_A^{-1}=J_\lambda$.
Different choices of $X_A$ may yield different subspaces of $C_{\M}(J_\lambda)$.
We must therefore specify~$X_A$ carefully.
We do this inductively, by first assuming that $A|_{n-1}=J_\mu$ for some partition $\mu\vdash (n-1)$.

We begin by giving an overview of the process by which we put $A$ into Jordan form.
This overview, along with the example computation below, should provide enough detail for the reader to understand the conjugation process.
For completeness, we also provide explicit definitions of the matrices used in the conjugation procedure.

\subsection*{Conjugation procedure}
\begin{enumerate}
    \item Use the non-zero entries from $A|_{n-1}=J_\mu$ to set as many entries as possible in column $n$ to zero.
        This can be achieved with a product of upper-triangular transvections.
        The resulting matrix will only have non-zero entries in column $n$ which are at the bottom of a block.
        That is, in cells of the form~$(n, \mu_1+\cdots+\mu_k)$.
        If the entire column $n$ is zero in the result, skip ahead to step \ref{conj_alg:perm_start}.
        In this case, the matrices by which we conjugate in the intermediate steps will all be defined as the identity.

    \item We may now assume that column $n$ is non-zero. Conjugate by a diagonal matrix which scales the last column and last row in such a way as to set the first non-zero entry in column $n$ to be 1.

    \item Use the first non-zero entry in column $n$ (which now contains the value 1) to set every other value in that column to zero.
        This is achieved via a product of lower-triangular matrices.
        Each such lower triangular matrix will fix the $J_\mu$ in the top-left corner, and set a single cell in column $n$ to be zero.
    \item\label{conj_alg:first_perm} Apply a permutation matrix to move column $n$ so that its non-zero value aligns with a Jordan block, effectively increasing the size of this block by~1.
    \item\label{conj_alg:perm_start} At this point the matrix is in Jordan form, modulo rearranging the blocks to be in descending order.
        At most one block must be moved to guarantee this ordering.
        The block that must be moved is the one whose size was increased (if we increased the size of a block at all).
        Call this block the \emph{current block}.
        If we did not increase the size of a block, then we created a new block of size 1.
        In this case, the newly created block will be called the current block.
        Apply a permutation matrix which moves the current block as far to the top-left as possible, while still guaranteeing that the blocks are in descending order.
\end{enumerate}

\begin{exmp}\label{exmp:conj}
    The following is a worked example of the conjugation procedure for a matrix $A\in\U_9(\Q)$.
    We use $\Q$ as the field for simplicity, though using a finite field does not present any extra difficulty.
    We write $A^{[i]}$ for the matrix obtained after step $i$.
\[\scalemath{0.75}{\hphantom{{}^{[1]}}A=}\scalemath{0.55}{\left(\begin{array}{ccc|cc|cc|c||c}
    & 1 &   &   &   &   &   &   & 3\\
    &   & 1 &   &   &   &   &   & 1\\
    &   &   &   &   &   &   &   & 0\\\hline
    &   &   &   & 1 &   &   &   & 2\\
    &   &   &   &   &   &   &   & 2\\\hline
    &   &   &   &   &   & 1 &   & 4\\
    &   &   &   &   &   &   &   & 0\\\hline
    &   &   &   &   &   &   &   & 1\\\hline\hline
    0 & 0 & 0 & 0 & 0 & 0 & 0 & 0 & 0
\end{array}\right)}
\hspace{4mm}
\scalemath{0.75}{A^{[1]}=}\scalemath{0.55}{\left(\begin{array}{ccc|cc|cc|c||c}
    & 1 &   &   &   &   &   &   & 0\\
    &   & 1 &   &   &   &   &   & 0\\
    &   &   &   &   &   &   &   & 0\\\hline
    &   &   &   & 1 &   &   &   & 0\\
    &   &   &   &   &   &   &   & 2\\\hline
    &   &   &   &   &   & 1 &   & 0\\
    &   &   &   &   &   &   &   & 0\\\hline
    &   &   &   &   &   &   &   & 1\\\hline\hline
    0 & 0 & 0 & 0 & 0 & 0 & 0 & 0 & 0
\end{array}\right)}
\hspace{4mm}
\scalemath{0.75}{A^{[2]}=}\scalemath{0.55}{\left(\begin{array}{ccc|cc|cc|c||c}
    & 1 &   &   &   &   &   &   & 0\\
    &   & 1 &   &   &   &   &   & 0\\
    &   &   &   &   &   &   &   & 0\\\hline
    &   &   &   & 1 &   &   &   & 0\\
    &   &   &   &   &   &   &   & 1\\\hline
    &   &   &   &   &   & 1 &   & 0\\
    &   &   &   &   &   &   &   & 0\\\hline
    &   &   &   &   &   &   &   & \tfrac{1}{2}\\\hline\hline
    0 & 0 & 0 & 0 & 0 & 0 & 0 & 0 & 0
\end{array}\right)}\]
\[\scalemath{0.75}{A^{[3]}=}\scalemath{0.55}{\left(\begin{array}{ccc|cc|cc|c||c}
    & 1 &   &   &   &   &   &   & 0\\
    &   & 1 &   &   &   &   &   & 0\\
    &   &   &   &   &   &   &   & 0\\\hline
    &   &   &   & 1 &   &   &   & 0\\
    &   &   &   &   &   &   &   & 1\\\hline
    &   &   &   &   &   & 1 &   & 0\\
    &   &   &   &   &   &   &   & 0\\\hline
    &   &   &   &   &   &   &   & 0\\\hline\hline
    0 & 0 & 0 & 0 & 0 & 0 & 0 & 0 & 0
\end{array}\right)}
\hspace{4mm}
\scalemath{0.75}{A^{[4]}=}\scalemath{0.55}{\left(\begin{array}{ccc|ccc|cc|c}
    & 1 &   &   &   & 0 &   &   & \\
    &   & 1 &   &   & 0 &   &   & \\
    &   &   &   &   & 0 &   &   & \\\hline
    &   &   &   & 1 & 0 &   &   & \\
    &   &   &   &   & 1 &   &   & \\
    0 & 0 & 0 & 0 & 0 & 0 & 0 & 0 & 0\\\hline
    &   &   &   &   & 0 &   & 1 & \\
    &   &   &   &   & 0 &   &   & \\\hline
    &   &   &   &   & 0 &   &   &
\end{array}\right)}
\hspace{4mm}
\scalemath{0.75}{A^{[5]}=}\scalemath{0.55}{\left(\begin{array}{ccc|ccc|cc|c}
    & 1 & 0 &   &   &   &   &   & \\
    &   & 1 &   &   &   &   &   & \\
    0 & 0 & 0 & 0 & 0 & 0 & 0 & 0 & 0\\\hline
    &   & 0 &   & 1 &   &   &   & \\
    &   & 0 &   &   & 1 &   &   & \\
    &   & 0 &   &   &   &   &   & \\\hline
    &   & 0 &   &   &   &   & 1 & \\
    &   & 0 &   &   &   &   &   & \\\hline
    &   & 0 &   &   &   &   &   &
\end{array}\right)}\]
    Note that $A^{[4]}=A^{[5]}$, but the explicitly shown zeros are in different locations.
    This is to emphasize that the last step in our conjugation procedure may fix the Jordan form (as it does in this case).
    However, when we apply the same procedure to $C_\U(A)$, this last action will often be non-trivial.
\end{exmp}

\begin{rmk}
    It is tempting to assume that, if $A$ is already in Jordan form, then the conjugating matrix $X_A$ will be the identity.
    This is not necessarily so.
    While many of the steps in the conjugation procedure are trivial, the permutation matrices in steps \ref{conj_alg:first_perm} and \ref{conj_alg:perm_start} need not be the identity.
    For example, if $A=\id_n$, then $X_A$ is the permutation matrix defined by the permutation $w:k\mapsto n+1-k$.
    It is true however that if $A$ is already in Jordan form, then $X_A$ will be a permutation matrix.
\end{rmk}

To be specific about this procedure, we now write down explicitly the matrices used in the conjugaction procedure.
We conjugate $A$ by a product of five matrices, one for each step in the conjugation procedure.
As was shown in Example \ref{exmp:conj}, for~$i=1,\dots,5$ we will use $A^{[i]}$ to denote the matrix obtained after step $i$.

First, conjugate by $E_A:=\prod_{i=1}^{n-1}\EE_{i+1,n}(A_{i,n})$, where~$\EE_{i,j}(\alpha)$ is the transvection having 1s on the diagonal, $\alpha$ in position $(i,j)$, and zeros everywhere else.
Conjugating $A$ by $E_A$ gives the matrix $A^{[1]}=E_AAE_A^{-1}$ satisfying \begin{enumerate}
        \item $A^{[1]}|_{n-1}=J_\mu$,
        \item each row of $A^{[1]}$ has at most one non-zero entry.
    \end{enumerate}

    Thus, the only non-zero entries in the last column of $A^{[1]}$ are in line with the bottom of a Jordan block of $J_\mu$.

    Second, let $x$ denote the first non-zero entry in column~$n$ of $A^{[1]}$, if it exists (and~$x=1$ otherwise).
    Define \[\Delta_A:=\diag(\underbrace{1,\dots,1}_{n-1},x).\]
    Conjugating $A^{[1]}$ by $\Delta_A$ yields the matrix $A^{[2]}=\Delta_AA^{[1]}\Delta_A^{-1}$ which has at most one non-zero entry in each row, and which has a 1 as its first non-zero entry in column~$n$ (if column $n$ is has any non-zero entries).

    Third, define a lower-triangular matrix $L_A$ which uses the first non-zero entry in the $n$th column to set all other entries in that column to zero.
    If column $n$ is already zero, then simply set $L_A$ to be the identity.
    Otherwise, let $\tilde\mu_s:=\sum_{i=1}^s\mu_i$.
    These numbers are the indices of rows which are at the bottom of Jordan blocks.
    Every non-zero entry in column $n$ of $A^{[2]}$ appears in such a row.
    Define $r$ to be the integer such that $\tilde\mu_r$ is the index of the row containing the first non-zero entry in the $n$th column of $A^{[2]}$,
    and for $\alpha\in\F_q$ define 
    \begin{equation}\label{eq:lower}
        F_{j,r}(\alpha):=1+\alpha\sum_{k=1}^{\mu_j}\scalebox{1.3}{$e$}_{\tilde\mu_{j-1}+k,\;\tilde\mu_r-\mu_j+k},
    \end{equation}
    where $e_{i,j}$ denotes the matrix with a 1 in position $(i,j)$ and zeros everywhere else.
    Left-multiplication by $F_{j,r}(\alpha)$ takes the last $\mu_j$ rows in the $r$th block (rows~$\tilde\mu_r-\mu_j+1$ through $\tilde\mu_r$), and adds them to the rows in the $j$th block (first scaling them by~$\alpha$).
    Right-multiplication by $F_{j,r}(\alpha)$ takes the $\mu_j$ columns in the $j$th block, and subtracts them from the last $\mu_j$ columns in the $r$th block (first scaling them by $\alpha$).
    Conjugating~$A^{[2]}$ by $F_{j,r}(-A^{[2]}_{n,\tilde\mu_j})$ will leave $A^{[2]}$ unchanged in every entry except in the entry indexed by $(n,\tilde\mu_j)$, which will be set to zero.
    We can therefore define \[L_A:=\prod_{j>r} F_{j,r}(-A^{[2]}_{n,\tilde\mu_j}),\] so that $A^{[3]}=L_AA^{[2]}L_A^{-1}$ has one non-zero entry in column $n$, and that value is 1.

    Fourth, conjugate by a permutation matrix to make column $n$ align with the correct block.
    We apply $\sigma_A:=(\tilde\mu_r+1,\tilde\mu_r+2,\dots,n)$.
    If column $n$ is zero in $A^{[3]}$ then no such permutation is necessary, and we set $\sigma_A$ to be the identity.
    Now $A^{[4]}=\sigma_A A^{[3]} \sigma_A^{-1}$ is the direct sum of Jordan blocks, though not necessarily in descending order.

    Lastly, we apply a permutation matrix $\tau_A$ which moves the block as close to the top-left as possible to put the blocks into descending order, in such a way as to preserve the relative order of all other Jordan blocks.\\

    We are now in a position to define our choice of conjugating matrix $X_A$, so that $X_AAX_A^{-1}$ is in Jordan form.
    Define $X_A$ recursively.
    For the unique $A\in\U_1$, we take $X_A=(1)$ to be the identity matrix.
    For $n>1$, let $B=A|_{n-1}$, and define~$A'=X_BAX_B^{-1}$, so that $A'|_{n-1}$ is in Jordan form.
    Then define
    \begin{align*}
        Y_{A'} &:= \tau_{A'}\sigma_{A'}L_{A'}\Delta_{A'}E_{A'},\text{ and}\\
        X_A &:= Y_{A'}X_B.
    \end{align*}
    Here we have implicitly identified $X_B\in\M_{(n-1)\times(n-1)}$ with $X_B\oplus1\in\M_{n\times n}$.
    By construction, we have $X_A A X_A^{-1}=J_\lambda$.

\section{Gap arrays}\label{sec:gap}
\subsection{Preliminary results}
\begin{lem}\label{lem:jordan_rank}
    Let $\lambda\vdash a$, $\mu\vdash b$, and let $T_{\lambda,\mu}:\M_{a\times b}\to \M_{a\times b}$ be defined by $T_{\lambda,\mu}(X)=J_\lambda X-XJ_\mu$. Then
    \[\dim\ker T_{\lambda,\mu}=\sum_{i,j}\min\{\lambda_i,\mu_j\}=\angs{\lambda',\mu'}.\]
    In particular, $\dim C_\M(J_\lambda)=\|\lambda'\|^2.$
\end{lem}

\begin{proof}
    To compute the rank, we first do so for two Jordan blocks $J_{(a)}$ and $J_{(b)}$.
    Note that $J_{(a)}X$ is the matrix obtained by removing the top row of $X$, shifting all other rows upwards by one, and adding a row of zeros at the bottom.
    Similarly,~$XJ_{(b)}$ is the matrix obtained by removing the rightmost column of $X$, shifting all other columns right by one, and adding a new column of zeros on the left.
    If~$J_{(a)}X=XJ_{(b)}$, then these conditions guarantee that
    \begin{enumerate}
        \item all diagonals (top-left to bottom-right) are constant, and
        \item if a diagonal does not touch both the topmost row and the rightmost column, then it must be zero.
    \end{enumerate}
    Moreover, these conditions exactly describe the solutions to $J_{(a)}X=XJ_{(b)}$.
    The dimension of this space is given by the quantity $\min\{a,b\}$, the number of diagonals which are not forced to be zero.

    In general, if either $\lambda\vdash a$ or $\mu\vdash b$ have more than one part, then we decompose~$\M_{a\times b}$ into smaller subspaces by $\M_{a\times b}\cong \bigoplus_{i,j}\M_{\lambda_i\times \mu_j}$, so that the action on~$\M_{\lambda_i\times\mu_j}$ becomes $T_{(\lambda_i),(\mu_j)}$.
    The dimension of $\ker T_{\lambda,\mu}$ is the sum of the dimensions of the kernels of these subspaces, so \[\dim\ker T_{\lambda,\mu}=\sum_{i,j}\min\{\lambda_i,\mu_j\}.\]

    To see that $\sum_{i,j}\min\{\lambda_i,\mu_j\}=\angs{\lambda',\mu'}$, we will show that both sides count the number of pairs of cells in the diagrams of $\lambda$ and $\mu$ (one from each diagram) which lie in the same column.

    On the one hand, we may specify the column from which we chose the cells first.
    Then any pair of elements from these columns will suffice.
    Therefore, we obtain~$\sum_i \lambda'_i\mu'_i=\angs{\lambda',\mu'}$.

    On the other hand, we may first chose the row to which each of the cells belong.
    If the cell in the diagram for $\lambda$ lies in row $i$, and the cell in the diagram for $\mu$ lies in row $j$, then there are $\min\{\lambda_i,\mu_j\}$ choices which place these cells in the same column.
    As these two expressions count the same quantity, we obtain \[\sum_{i,j}\min\{\lambda_i,\mu_j\}=\angs{\lambda',\mu'}.\]

    Lastly, taking $\lambda=\mu$, we see that $\ker T_{\lambda,\lambda}$ is the centralizer of $J_\lambda$, and therefore \[\dim C_\M(J_\lambda)=\dim\ker T_{\lambda,\lambda}=\|\lambda'\|^2,\] which completes the proof.
\end{proof}

\begin{rmk}
    Matrices which satisfy the first condition (constant on top-left to bottom-right diagonals) are called \emph{Toeplitz} matrices.
    The matrices in $\ker T_{\lambda,\mu}$ are not only Toeplitz on each block, but are also upper-triangular.
    Specifically, any diagonal which does not touch the rightmost column and topmost row of a block is necessarily zero.
\end{rmk}

\subsection{Defining gap arrays}
In this section, we introduce a new combinatorial object called a \emph{gap array}.
Gap arrays are used to encode certain subspaces of $C_\M(J_\lambda)$.
Many algebraic operations we can apply to these subspaces are encoded succinctly by combinatorial operations we can apply to gap arrays.
This will be our primary tool for computing upper bounds on $\kk(U_n)$.

\begin{defn}
    Let $\lambda\vdash n$ be a partition of length $\ell:=\ell(\lambda)$.
    A \emph{gap array} of type~$\lambda$ is an $\ell\times\ell$ matrix $\G=(\G_{i,j})$ of integers satisfying 
    \begin{equation}\label{eq:gap_def}
        \max\{0, \lambda_i - \lambda_j\} \le \G_{i,j} \le \lambda_i.
    \end{equation}
\end{defn}

We use gap arrays to define subspaces of $C_{\M}(J_\lambda)$.
Recall that $C_{\M}(J_\lambda)$ has a natural block decomposition into blocks of $(\lambda_i\times\lambda_j)$-submatrices.
The block with rows corresponding to the $i$th part of $\lambda$ and columns corresponding to the $j$th part of $\lambda$ is called the \emph{$(i,j)$-block}.
Specifically, it consists of all cells $(x,y)$ in the matrix such that 
    \begin{align*}
        \lambda_1+\cdots+\lambda_{i-1} &< x \le \lambda_1+\cdots+\lambda_i,\text{ and}\\
        \lambda_1+\cdots+\lambda_{j-1} &< y \le \lambda_1+\cdots+\lambda_j.
    \end{align*}

    \begin{defn}
        For a gap array $\G$ of type $\lambda$, we let $C(\G)$ denote the subspace of $C_{\M}(J_\lambda)$ satisfying the condition that on the $(i,j)$-block, the lowest $\G_{i,j}$ diagonals touching the right boundary are all zero.
    \end{defn}

\begin{exmp}
    Let $\lambda=(6,2,1,1)\vdash 10$, and let $\G$ denote the gap array \[\G=\begin{pmatrix}
            2 & 4 & 6 & 5\\
            1 & 0 & 1 & 2\\
            0 & 1 & 1 & 0\\
            1 & 0 & 0 & 0
    \end{pmatrix}.\]

    Then the corresponding subspace $C(\G)$ of $C_{\M}(J_\lambda)$ consists of all $10\times 10$ matrices of the form shown below (zeros omitted).
    In this example, $C(\G)$ is a 16-dimensional subspace of $C_{\M}(J_\lambda)$.

    \[\left(\begin{array}{cccccc|cc|c|c}
            &   & a_3 & a_4 & a_5 & a_6 & b_5 & b_6 &     & c_6\\
            &   &     & a_3 & a_4 & a_5 &     & b_5 &     & \\
            &   &     &     & a_3 & a_4 &     &     &     & \\
            &   &     &     &     & a_3 &     &     &     & \\
            &   &     &     &     &     &     &     &     & \\
            &   &     &     &     &     &     &     &     & \\\hline

            &   &     &     &     & d_2 & e_1 & e_2 & f_2 & \\
            &   &     &     &     &     &     & e_1 &     & \\\hline

            &   &     &     &     & g_1 &     &     &     & h_1\\\hline

            &   &     &     &     &     &     & i_1 & j_1 & k_1
    \end{array}\right)\]
\end{exmp}

\begin{rmk}
    The name ``gap array'' is chosen, because the recorded numbers measure the gap between the bottom of a block and the triangle it contains.
    Recall from the proof of Lemma \ref{lem:jordan_rank} that a $C_{\M}(J_\lambda)$ is characterized by, being Toeplitz on each block, and having any diagonal which does not touch both the topmost row and rightmost column of a block set to zero.
    Thus, $C_{\M}(J_\lambda)$ can be expressed as~$C(\G)$, where $\G_{i,j}=\max\{0,\lambda_i-\lambda_j\}$.
    At the other extreme, the zero subspace of~$\M_{n\times n}$ is $C(\G)$, where $\G_{i,j}=\lambda_i$.
    The inequalities in \eqref{eq:gap_def} which defines gap arrays encode these observations.
\end{rmk}

For gap arrays $\G=(\G_{i,j})$ and $\H=(\H_{i,j})$ of the same type, we say that $\G\le \H$ if $\G_{i,j}\le \H_{i,j}$ for all $i$ and $j$.
It follows that $\G\le \H$ if and only if $C(\H)\subseteq C(\G)$.
We also write $\abs{\G}=\sum_{i,j}\G_{i,j}$.

\begin{prop}\label{prop:gap_dim}
    Let $\G$ be a gap array of type $\lambda\vdash n$, and let $\ell=\ell(\lambda)$.
    Then \[\dim C(\G)=n\ell-\abs{\G}.\]
\end{prop}

\begin{proof}
    Observe that $\G_{i,j}$ counts the number of diagonals touching the rightmost column of the $(i,j)$-block which are zero in every element of $C(\G)$.
    Thus, $\abs{\G}$ counts the number of diagonals in all blocks which are necessarily zero.
    The quantity $\dim C(\G)$ counts the number of diagonals in all blocks which may be non-zero.
    This constitutes all diagonals, so $\abs{\G}+\dim C(\G)$ is a constant depending only on the type~$\lambda$ of $\G$.
    Taking $\dim C(\G)=0$, by setting $\G_{i,j}=\lambda_i$, we see that~$\abs{\G}=\sum_{i,j}\lambda_i=n\ell$.
    Hence, for every gap arary $\G$ of type $\lambda$,
    \begin{equation}\label{eq:gap_dim}
        \dim C(\G)=n\ell-\abs{\G}.
    \end{equation}
\end{proof}

\subsection{Combinatorial operations on gap arrays}\label{subsec:ops}
Recall from Section \ref{sec:jordan_forms_and_conj} that we can put a matrix $A$ into Jordan form by iteratively conjugating larger and larger submatrices into Jordan form.
As we do so, we conjugate $C_\U(A)$ by the same process.
The purpose of this section is to translate what happens to $C_\U(A)$ in this iterative procedure into the combinatorics of gap arrays.

For a partition $\lambda$, and for $r\in\{1,\dots,\ell(\lambda)\}$, let $\phi_r(\lambda)$ denote the partiton obtained by adding a block to the diagram for $\lambda$ into row $r$.
If $\lambda_r=\lambda_{r-1}$, then the result will be a composition and no longer a partition, so we reorder the rows to obtain a partition.
For example, if $\lambda=(3^12^21^4)$, then $\phi_6(\lambda)=(3^12^31^3)$.
If~$r=\ell(\lambda)+1$, then we define $\phi_r(\lambda)$ to be the partition obtained by adding a part of size 1.

For a gap array $\G$ of type $\lambda=(\lambda_1,\dots,\lambda_\ell)$, we define a new gap array $\psi_r(\G)$ of type $\phi_r(\lambda)$.
Define $\psi_r(\G)$ as follows:
\begin{enumerate}
    \item If $r=\ell+1$, add a new row of all zeros at the bottom of $\G$, and a new column at the right of $\G$ with the values $(\lambda_1,\dots,\lambda_\ell,0)$.
    \item Subtract 1 from each non-zero entry in column $r$ of $\G$.
    \item Add 1 to each entry in row $r$.
    \item Permute the rows and columns so that corresponding block sizes are in decreasing order.
        Specifically, move row and column $r$ to be the very first row and column of their particular block size. Keep all other rows/columns in the same relative order.
\end{enumerate}

\begin{prop}
    Let $\G$ be a gap array of type $\lambda$, and let $r\in\{1,\dots,\ell(\lambda)+1\}$. Then $\psi_r(\G)$ is a gap array of type $\phi_r(\lambda)$.
\end{prop}
\begin{proof}
    First note that to obtain $\psi_r(\G)$, the rows and columns of $\G$ are rearranged according to the same permutation as the entries in $\lambda$ when computing $\phi_r(\lambda)$.
    We can therefore ignore the permutations and check that that the entries in $\psi_r(\G)$ satisfy the inequalities given in \eqref{eq:gap_def}.

    Note that $\G_{i,j}$ will be increased (by one) if and only if $i=r$.
    Because $\lambda_r$ is also increased by one, the upper bound given in \eqref{eq:gap_def} will still be satisfied.
    From the definition of $\psi_r$, we see that no entry in $\psi_r(\G)$ can be negative.
    Hence, the only way an inequality from \eqref{eq:gap_def} can fail to be satisfied is if $\G_{i,j}$ is decreased (by one).
    However, in this situation, it must be that $j=r$, and so the lower bound of $\lambda_i-\lambda_j$ is also decreased by one.
    Thus, the inequalities from \eqref{eq:gap_def} are always preserved, and so $\psi_r(\G)$ is a gap array of type $\phi_r(\G)$.
\end{proof}

\begin{exmp}
    Starting with $\lambda=(3,2,1)\vdash 6$, we apply $\phi_2$, $\phi_4$, and $\phi_4$ in that order.
    We obtain 
    \begin{align*}
        \phi_2(\lambda) &= (3,3,1)\\
        \phi_4(\phi_2(\lambda))&= (3,3,1,1)\\
        \phi_4(\phi_4(\phi_2(\lambda))) &= (3,3,2,1).
    \end{align*}
    Correspondingly, for a gap array, let \[\G=\begin{pmatrix}
            1 & 1 & 2\\
            1 & 1 & 1\\
            0 & 0 & 1
    \end{pmatrix}.\]
    Applying $\psi_2$, $\psi_4$, and $\psi_4$ in that order to $\G$, we obtain
    \[\psi_2(\G)=\begin{pmatrix}
            1 & 2 & 2\\
            0 & 1 & 2\\
            0 & 0 & 1
        \end{pmatrix},\hspace{5mm}
        \psi_4(\psi_2(\G))=\begin{pmatrix}
            1 & 2 & 2 & 2\\
            0 & 1 & 2 & 2\\
            1 & 1 & 1 & 1\\
            0 & 0 & 0 & 1
    \end{pmatrix},\text{ and}\]
    \[\psi_4(\psi_4(\psi_2(\G)))=\begin{pmatrix}
            1 & 2 & 1 & 2\\
            0 & 1 & 1 & 2\\
            1 & 1 & 1 & 1\\
            1 & 1 & 0 & 1
    \end{pmatrix}.\]

\end{exmp}

\begin{defn}
    Let $\lambda\vdash n$, and let $r\in\{1,\dots,\ell(\lambda)\}$.
    Then a gap array $\G$ of type $\lambda$ is called \emph{$r$-valid} if the following two conditions hold
    \begin{enumerate}
        \item Row $r$ is element-by-element weakly larger than every row below it.
            That is, for every $j>r$, and every $k$, \[\G_{j,k}\le \G_{r,k}.\]
        \item Column $r$ is element-by-element weakly smaller than every column to the right.
            That is, for every $j>r$ and every $k$, \[\G_{k,r}\le \G_{k,j}.\]
    \end{enumerate}
\end{defn}
Validity is a technical condition which will be needed in Lemma \ref{lem:induct_gap_build}.
The most important property to recognize is that a gap array is $r$-valid for every $r$ if and only if its entries are increasing from left to right and from bottom to top.

We need one more definition to state Lemma \ref{lem:induct_gap_build}.
For a subspace $V\subseteq \M_{(n-1)\times(n-1)}$, define 
\[\overline{V}:=\{X\in\M_{n\times n}\mid X|_{n-1}\in V, X_{n,i}=0\text{ for all }i\le n-1\}.\]
Graphically, we are considering the subspace shown in Figure \ref{fig:overline_space}.
We are now ready to state and prove our main lemma regarding gap arrays.
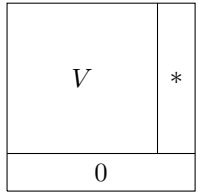
\begin{figure}[h]
    \begin{tikzpicture}[scale=0.5]
        \draw (0, 0) -- (0, 5) -- (5,5) -- (5,0) -- (0,0);
        \draw (4,5) -- (4,1);
        \draw (0, 1) -- (5,1);
        \node at (2.5, 0.5) {0};
        \node at (4.5, 3) {$\ast$};
        \node at (2, 3) {$V$};
    \end{tikzpicture}
    \caption{A graphical representation of $\overline{V}$, where $V\subseteq\M_{(n-1)\times(n-1)}$.}
    \label{fig:overline_space}
\end{figure}

\begin{lem}\label{lem:induct_gap_build}
    Let $A\in\U_n$ and let $\mu\vdash(n-1)$ such that $A|_{n-1}=J_\mu$.
    Let $\G$ be an $r$-valid gap array of type $\mu$.
    Then \[Y_A \left(\overline{C(\G)}\cap C_{\M}(A)\right) Y_A^{-1}=C(\psi_r(\G)).\]
\end{lem}

\begin{proof}
    We first consider $Y_A\overline{C(\G)}Y_A^{-1}$.
    Note that for any subspace $V\subseteq \M_{(n-1)\times(n-1)}$, the matrix $E_A$ normalizes $\overline V$, as its action only changes the last column, and fixes the bottom-right entry.
    Similarly, $\Delta_A$ normalizes $\overline V$. 
    Hence, \[Y_A\overline{C(\G)}Y_A^{-1}=(\tau_A\sigma_AL_A)\overline{C(\G)}(\tau_A\sigma_AL_A)^{-1}.\]
    
    Next, we claim that $L_A$ normalizes $\overline{C(\G)}$.
    Let $A_{[i,j]}$ denote the $(\mu_i\times\mu_j)$-matrix obtained by restricting $A$ to the $(i,j)$-block.
    Recall that $L_A$ is defined in \eqref{eq:lower} to be the product of the lower-triangular matrices of the form $F_{j,r}(\alpha)$.
    Let $B\in \overline{C(\G)}$. Then by construction, $F_{j,r}(\alpha)\cdot B$ agrees with $B$ on all blocks except for those of the form $B_{[j,k]}$ for some $k$.
    On such blocks, the action of left-multiplication by $F_{j,r}(\alpha)$ is to add the bottom $\lambda_j$ rows of $\alpha\cdot B_{[r,k]}$ to $B_{[j,k]}$.
    Because $\G$ is $r$-valid, and $j>r$, we know that $\G_{j,k}\le \G_{r,k}$ for all $k$.
    Thus, every diagonal of non-zero entries in $\alpha\cdot B_{[r,k]}$ gets added to a diagonal in $B_{[j,k]}$ which is allowed to be non-zero.
    As a result, left-multiplication by $F_{j,r}(\alpha)$ stabilizes $\overline{C(\G)}$.

    This is shown graphically by the left diagram in Figure \ref{fig:validity_exmp}.
    The light-gray strips represent the rows and columns of the relevant blocks.
    The vertical strip represents the $(\ast,k)$-blocks, and the horizontal strips represent the $(r,\ast)$- and $(j,\ast)$-blocks.
    The dark-gray triangles represent those diagonals in each block which may be non-zero.
    The black triangle shows the cells in the $(r,k)$-block which get carried to the $(j,k)$-block via the left-multiplication by $F_{j,r}(\alpha)$.

    \begin{figure}[h]
        \centering
        \begin{tikzpicture}[scale=0.04]
            \definecolor{light-gray}{gray}{0.85}

            \draw (0, 0) -- (0, 100) -- (100, 100) -- (100, 0) -- (0, 0);

            \draw[fill=light-gray] (50, 0) rectangle (70, 100);
            \draw[fill=light-gray] (0, 7) rectangle (100, 22);
            \draw[fill=light-gray] (0, 60) rectangle (100, 83);
            \draw (50, 0) -- (50, 100);
            \draw (70, 0) -- (70, 100);
 
            \draw[fill=gray] (70, 67) -- (70, 83) -- (54, 83) -- (70, 67);
            \draw[fill=black] (70, 67) -- (70, 75) -- (62, 75) -- (70, 67);
            \draw[fill=gray] (70, 12) -- (70, 22) -- (60, 22) -- (70, 12);
            
            \draw [decorate,decoration={brace,amplitude=1pt},xshift=1pt,yshift=0pt] (69, 61) -- (69, 66);
            \node at (61, 63) {\tiny$\G_{r,k}$};

            \draw [decorate,decoration={brace,amplitude=1pt},xshift=1pt,yshift=0pt] (71,11.5) -- (71, 7.5);
            \node at (80, 9.5) {\tiny$\G_{j,k}$};

            \draw [decorate,decoration={brace,amplitude=2pt},xshift=2pt,yshift=0pt] (71,74) -- (71, 61);
            \node at (79, 67) {\tiny$\lambda_j$};
 
            \node at (57, 72) {\tiny0};
            \node at (60, 15) {\tiny0};
        \end{tikzpicture}
        \begin{tikzpicture}[scale=0.04]
            \definecolor{light-gray}{gray}{0.85}

            \draw (0, 0) -- (0, 100) -- (100, 100) -- (100, 0) -- (0, 0);

            \draw[fill=light-gray] (0,50) rectangle (100, 70);
            \draw[fill=light-gray] (7,0) rectangle (30,100);
            \draw[fill=light-gray] (60,0) rectangle (75,100);

            \draw[fill=gray] (14, 70) -- (30, 70) -- (30, 54) -- (14, 70);
            \draw[fill=black] (20, 70) -- (30, 70) -- (30, 60) -- (20, 70);
            \draw[fill=gray] (65, 70) -- (75, 70) -- (75, 60) -- (65, 70);
            \draw (0, 50) -- (100, 50);
            \draw (0, 70) -- (100, 70);

            \draw [decorate,decoration={brace,amplitude=1pt},xshift=1pt,yshift=0pt] (29,50.5) -- (29, 54);
            \node at (20, 52.8) {\tiny$\G_{k,r}$};

            \draw [decorate,decoration={brace,amplitude=2pt},xshift=2pt,yshift=0pt] (31,59) -- (31, 50.5);
            \node at (41, 54) {\tiny$\G_{k,j}$};

            \draw [decorate,decoration={brace,amplitude=2pt},xshift=2pt,yshift=0pt] (76,59) -- (76, 50.5);
            \node at (86, 54) {\tiny$\G_{k,j}$};
            \node at (67, 57) {\tiny0};
            \node at (15, 60) {\tiny0};
        \end{tikzpicture}

        \caption{A graphical representation of how $r$-validity implies that conjugation by $F_{j,r}(\alpha)$ stabilizes $\overline{C(\G)}$.}
        \label{fig:validity_exmp}
    \end{figure}
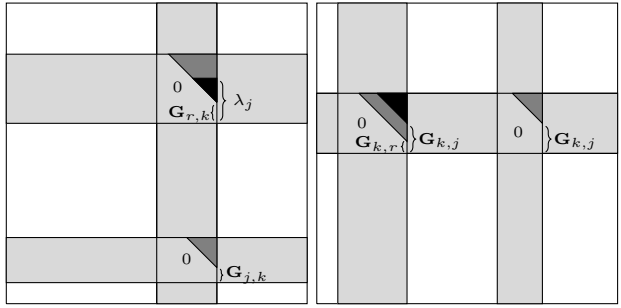

    Similarly, right-multiplication by $F_{j,r}(\alpha)^{-1}$ maps $\overline{C(\G)}$ to itself as a consequence of the column-condition of $r$-validity.
    For each $k$, the action of right-multiplication by~$F_{j,r}(\alpha)^{-1}$ subtracts $\alpha\cdot B_{[k,j]}$ from $B_{[k,r]}$, where the right edges of each block are aligned, and $\alpha\cdot B_{[k,j]}$ is extended to the left with zeros if necessary.
    Because $\G$ is $r$-valid, and~$j>r$, we know that $\G_{k,r}\le \G_{k,j}$ for all $k$.
    Thus, every diagonal of non-zero entries in $B_{[k,j]}$ gets subtracted from a diagonal in $B_{[k,r]}$ which is allowed to be non-zero.

    This is shown graphically by the diagram on the right in Figure \ref{fig:validity_exmp}.
    The dark-gray triangles represent those diagonals in each block which may be non-zero.
    The black triangle shows where the right-multiplication by $F_{j,r}(\alpha)^{-1}$ carries the triangle from the $(k,j)$-block.
    Because~$\G_{k,j}\ge \G_{k,r}$, the black triangle lies inside the dark-gray triangle in the $(k,r)$-block.
    Hence, right-multiplication by $F_{j,r}(\alpha)^{-1}$ stabilizes $\overline{C(\G)}$.
    In this way, we see that $r$-validity is a combinatorial description of the fact that~$F_{j,r}(\alpha)$ normalizes~$\overline{C(\G)}$.
    It follows that $Y_A\overline{C(\G)}Y_A^{-1}=(\tau_A\sigma_A)\overline{C(\G)}(\tau_A\sigma_A)^{-1}$.

    Next, recognize that $\sigma_A$ moves the last row and column into the $r$th block, effectively increasing the size of the $r$th block by 1.
    Then $\tau_A$ acts by rearranging these blocks, guaranteeing that their sizes are in descending order.
    These permutations have the same action on the block sizes as the function $\phi_r$ does to the parts of $\mu$.
    Thus, the block sizes are now described by the partition $\phi_r(\mu)$.
    Hence, \begin{equation}\label{eq:subspace}
        Y_A\left(\overline{C(\G)}\cap C_\M(A)\right)Y_A^{-1}=(\tau_A\sigma_A)\overline{C(\G)}(\tau_A\sigma_A)^{-1}\cap C_\M(J_{\phi_r(\mu)}).
    \end{equation}

    The left-hand side of \eqref{eq:subspace} is a subspace of $C_\M(J_{\phi_r(\mu)})$.
    Therefore every block must be Toeplitz (constant on diagonals), and any diagonal not touching both the topmost row and rightmost column of a block must be zero.
    We now consider how the action of $\tau_A\sigma_A$ affects each block.
    Because $\tau_A$ only permutes the blocks, we focus our attention on the action of $\sigma_A$.
    If $i\ne r$ and $j\ne r$, then $\sigma_A$ has no affect on an $(i,j)$-block.
    This coincides with the fact that $\psi_r$ does not affect cells $(i,j)$ in a gap array when $i\ne r$ and $j\ne r$.

    Considering an $(r,j)$-block for $j\ne r$, we see that $\sigma_A$ moves the last row to the bottom of such a block.
    As this row is necessarily zero, it increases the gap on an $(r,j)$-block by 1.
    This is precisely what the map $\psi_r$ encodes by adding 1 to each entry in row $r$ of the gap array.

    Considering an $(i,r)$-block, for $i\ne r$, we see that $\sigma_A$ moves the last column to the right edge of this block.
    Because the result is guaranteed to be Toeplitz, the size of the gap can only decrease by 1, as shown in Figure~\ref{fig:decr_by_one}.
    If the gap size on an $(i,r)$-block was already zero in $C(\G)$, then such matrices cannot possibly lie in~$Y_A\left(\overline{C(\G)}\cap C_\M(A)\right)Y_A^{-1}$, because the action of $\sigma_A$ would force the $(i,r)$-block to no longer satisfy the upper-triangular conditions necessary to lie in $C_\M(J_{\phi_r(\mu)})$.
    Thus, the resulting block will have its gap size decreased by 1 unless it is already zero, in which case, the gap size remains zero.
    This is encoded by the action of $\psi_r$ on the $r$th column of~$\G$.
    \begin{figure}
        \begin{tikzpicture}[scale=0.03]
            \definecolor{light-gray}{gray}{0.85}
            \draw[fill=lightgray] (100, 20) -- (90, 20) -- (90, 30) -- (80, 30) -- (80, 40) -- (70, 40) -- (70, 50) -- (60, 50) -- (60, 60) -- (90, 60) -- (90, 20);
            \draw[fill=gray] (100, 60) -- (90, 60) -- (90, 20) -- (100, 20) -- (100, 60);
            \draw (0, 0) -- (0, 60) -- (100, 60) -- (100, 0) -- (0,0);
            \draw (90, 0) -- (90, 60);
            \draw [decorate,decoration={brace,amplitude=2pt},xshift=2pt,yshift=0pt] (88, 1) -- (88, 29);
            \node at (75, 15) {\tiny$\G_{i,r}$};
            \draw [decorate,decoration={brace,amplitude=2pt},xshift=2pt,yshift=0pt] (102, 19) -- (102, 1);
            \node at (123, 9) {\tiny$\G_{i,r} - 1$};

        \end{tikzpicture}
        \caption{A graphical representation of an $(i,r)$-block is affected by the action of $\sigma_A$ (for $i\ne r$).}
        \label{fig:decr_by_one}
    \end{figure}
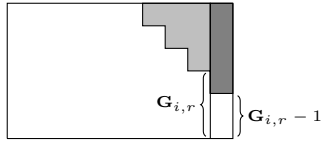

    Regarding the $(r,r)$-block, both its width and height are increased by 1.
    As the resulting space is guaranteed to be Toeplitz, the gap size does not change.
    This is consistent with $\psi_r(\G)$. Thus, it follows that $Y_A\left( \overline{C(\G)}\cap C_\M(A) \right)=C(\psi_r(\G))$ as desired.
\end{proof}

\section{Proof of Theorem \ref{thm:conj_bound}}\label{sec:proof}
\subsection{Overview of proof methods for Theorem \ref{thm:conj_bound}.}
In this section, we prove Theorem \ref{thm:better_conj_bound}, which trivially implies Theorem \ref{thm:conj_bound}.
Our proof proceeds as follows.
We first construct, for each $\lambda\vdash n$, a gap array $\G^\lambda$ that encodes a particular subspace of~$C_\M(J_\lambda)$.
The space $C(\G^\lambda)$ contains every space of the form $X_AC_\U(A)X_A^{-1}$, where $\shape(A)=\lambda$.
This is statement is made precise in Theorem~\ref{thm:worst_gap}.
The gap array $\G^\lambda$ is particularly easy to analyze combinatorially.
We compute $\abs{\G^\lambda}$ in Proposition~\ref{prop:g_bound}.
We then show that the worst possible choice $\lambda$ gives an exponent of $\frac{1}{6}n^2$.
The content of the proof of Theorem~\ref{thm:better_conj_bound} is to combine this technique with a simpler bound in order to improve this exponent.

Recall the notation $\comm(\U_n)=\{(A,B)\in\U_n\times\U_n\mid AB=BA\}$.
We introduce two families of subspaces of $\comm(\U_n)$ which we parameterize with partitions.
For partitions $\lambda,\mu\vdash n$, define 
\begin{align*}
    \comm(\lambda) &:=\{(A,B)\in\U_n\times\U_n\mid AB=BA,\ \shape(A)=\lambda\},\text{ and}\\
    \comm(\lambda,\mu) &:=\{(A,B)\in\U_n\times\U_n\mid AB=BA,\ \shape(A)=\lambda,\ \shape(B)=\mu\}.
\end{align*}

We improve on the $\tfrac{1}{6}n^2$ exponent by first proving a technical lemma (Lemma~\ref{lem:asymptotic_technicalities}) regarding the bounds on each $\comm(\lambda)$.
This lemma allows us to determine when gap arrays will be useful, and when we should bound $\kk(U_n)$ via different, but simpler, techniques.

\subsection{Bounds on gap arrays}
For any partition $\lambda\vdash n$, there is a specific gap array whose corresponding subspace of $C_{\M}(J_\lambda)$ contains all possible subspaces coming from upper-triangular matrices.
Moreover, we can construct this gap array explicitly
\begin{defn}
    For $\lambda\vdash n$, we define the gap array $\G^\lambda$ by
    \[\G^\lambda_{i,j}=\begin{cases}
            \lambda_i-\lambda_j & \text{if }\lambda_i>\lambda_j\\
            1 & \text{if }\lambda_i=\lambda_j\text{ and }i\le j\\
            0 & \text{otherwise.}
    \end{cases}\]
\end{defn}

\begin{exmp}
    For the partition $\lambda=(6,3,1,1,1)$, we have 
    \[\G^\lambda=\begin{pmatrix}
            1 & 3 & 5 & 5 & 5\\
            0 & 1 & 2 & 2 & 2\\
            0 & 0 & 1 & 1 & 1\\
            0 & 0 & 0 & 1 & 1\\
            0 & 0 & 0 & 0 & 1
    \end{pmatrix}.\]
\end{exmp}

\begin{prop}\label{prop:worst}
    For every $\lambda\vdash n$, the gap array $\G^\lambda$ is $r$-valid for every $r$, and \[\G^{\phi_r(\lambda)}\le\psi_r(\G^\lambda).\]
\end{prop}

\begin{proof}
    A gap array is $r$-valid for every $r$ if and only if the entries are weakly increasing from left to right and from bottom to top.
    From the definition of $\G^\lambda$, it is clear that~$\G^\lambda$ satisfies this property.
    It therefore suffices to prove that $\G^{\phi_r(\lambda)}\le\psi_r(\G^\lambda)$.

    First suppose that $r=\ell+1$, where $\ell=\ell(\lambda)$.
    Recalling the definition of $\psi_r$ from Section~\ref{subsec:ops}, we see that that $\psi_r(\G^\lambda)$ adds a new column containing the value $\lambda_i-1$ in entry~$(i,\ell+1)$, and a new row of all 1s in row $\ell+1$.
    This row and column are then permuted to be the first row and column corresponding to a part of $\phi_r(\lambda)$ of size 1.
    We may obtain $\G^{\phi_r(\lambda)}$ can be obtained from $\G^\lambda$ by adding a new column containing the number $\lambda_i-1$ in the cell associated to the $i$th part of $\lambda$, and a new row containing 1s and 0s.
    The new columns and rows in these two gap arrays align by construction.
    If the new values in these rows and columns do not agree, it is because a cell has the value 0 in $\G^{\phi_r(\lambda)}$, and the value 1 in $\psi_r(\G^\lambda)$.
    Thus, we obtain $\G^{\phi_r(\lambda)}\le \psi_r(\G^\lambda)$.

    Now suppose $1\le r\le \ell(\lambda)$.
    The map $\psi_r$ acts by subtracting~1 from each non-zero entry in column~$r$, adding 1 to each entry in row~$r$, and then permuting the rows and columns so that this row and column are the first corresponding to a part of size $\lambda_r+1$.
    Let $w$ denote this permutation.
    Let $s$ denote the new index of row/column $r$ after this permutation, and let $\H=\psi_r(\G^\lambda)$.
    Then from the definition of the map $\psi_r$ in Section~\ref{subsec:ops}, we see that
    \begin{enumerate}
        \item $\H_{w(r),w(r)}=1$,
        \item $\H_{w(r),w(j)}=\G^{\lambda}_{r,j}+1$ for $j\ne r$,
        \item $\H_{w(i),w(r)}=\max\{0,\G^{\lambda}_{i,r}-1\}$ for $i\ne r$,
        \item $\H_{w(i),w(j)}=\G^{\lambda}_{i,j}$ for all $i,j\ne r$.
    \end{enumerate}

    Now consider $\G^{\phi_r(\lambda)}$.
    Recall that $\phi_r$ acts on $\lambda$ by increasing the $r$th part, and then permuting the parts so that the resulting composition becomes a permutation again.
    Specifically, this can be done by applying the same permutation $w$ as above.
    In other words, let $\mu:=\phi_r(\lambda)$. Then
    \[\mu_{w(i)}=\begin{cases}
            \lambda_r+1 & \text{if }i=r\\
            \lambda_i & \text{if }i\ne r.
    \end{cases}\]

    We now show that for all $i$ and $j$, we have~$\G^{\mu}_{w(i),w(j)}\le \H_{w(i),w(j)}$.
    We break this into four cases to coincide with the four items defining $\H_{w(i),w(j)}$ above.

    \subsection*{Case 1: $i=j=r$}\hspace{1mm}\\
    It is immediate from the definition of $\G^{\mu}$ that every entry on the main diagonal is equal to 1.
    Thus,
    \[\G^{\mu}_{w(r),w(r)}=1=\H_{w(r),w(r)}.\]

    \subsection*{Case 2: $i=r$, $j\ne r$}\hspace{1mm}\\
    If $\mu_{w(r)}>\mu_{w(j)}$, then $\G^{\mu}_{w(r),w(j)}=\mu_{w(r)}-\mu_{w(j)}$.
    We therefore obtain \begin{align*}
        \G^{\mu}_{w(r),w(j)} &= \mu_{w(r)}-\mu_{w(j)}\\
        &= \lambda_r+1-\lambda_j = \G^{\lambda}_{r,j}+1 = \H^{\lambda}_{w(r),w(j)}.
    \end{align*}

    Otherwise, we see that $\G^{\mu}_{w(r),w(j)}\le 1$, and $\H_{w(r),w(j)}=\G^{\lambda}_{r,j}+1\ge 1$.
    Thus, we have $\G^{\mu}_{w(r),w(j)}\le \H_{w(r),w(j)}$ as desired.

    \subsection*{Case 3: $i\ne r$, $j=r$}\hspace{1mm}\\
    First suppose $\mu_{w(i)}>\mu_{w(r)}$, so that $\G^{\mu}_{w(i),w(r)}=\mu_{w(i)}-\mu_{w(r)}$.
    Thus, we see that~$\lambda_i-(\lambda_r+1)>0$, so $\H_{w(i),w(r)}=\G^{\lambda}_{i,r}-1$.
    We therefore obtain 
    \[\G^{\mu}_{w(i),w(r)} = \mu_{w(i)}-\mu_{w(r)} = \lambda_i-(\lambda_r+1) = \H_{w(i),w(r)}.\]

    Next, suppose that $\mu_{w(i)}=\mu_{w(r)}$, and $w(i)\le w(r)$.
    Recall that $w$ is defined to have the property that $w(r)$ is the first part in $\mu$ of size $\mu_{w(r)}$.
    It follows that if~$\mu_{w(i)}=\mu_{w(r)}$ and $w(i)\le w(r)$, then in fact $w(i)=w(r)$.
    This is not possible, as~$i\ne r$ by hypothesis.

    In the remaining cases, $\G^{\mu}_{w(i),w(r)}=0$.
    We know from the definition of gap arrays that $\H_{w(i),w(r)}\ge 0$ as desired.

    \subsection*{Case 4: $i\ne r$, $j\ne r$}\hspace{1mm}\\
    Recalling that for $i\ne r$, we have $\mu_{w(i)}=\lambda_i$, we see that
    \[\G^{\mu}_{w(i),w(j)}=\G^{\lambda}_{i,j}=\H_{w(i),w(j)}.\]

    Thus, for every $i$ and $j$, we have $\G^\mu_{w(i),w(j)}\le \H_{w(i),w(j)}$, so we obtain the relationship $\G^{\phi_r(\lambda)}\le \psi_r(\G^\lambda)$ as desired.
\end{proof}

\begin{prop}\label{prop:g_bound}
    For any $\lambda\vdash n$ with $\ell=\ell(\lambda)$, we have 
    \[\abs{\G^\lambda}=n\ell-n-2\n(\lambda)+\sum_i\frac{m_i^2}{2}+\frac{\ell}{2}.\]
\end{prop}
\begin{proof}
    Recall that $\G^\lambda$ is an upper-triangular array, and that for $i\le j$ and $\lambda_i=\lambda_j$, we have~$\G^\lambda_{i,j}=1$.
    Such cells in $\G^\lambda$ contribute $\sum_i\binom{m_i+1}{2}$.
    For $\lambda_i>\lambda_j$, we have~$\G^\lambda_{i,j}=\lambda_i-\lambda_j$.
    Altogether, we have \[\abs{\G^\lambda}=\sum_{i>j}\lambda_i-\sum_{i>j}\lambda_j+\sum_i\binom{m_i+1}{2}.\]
    To compute $\sum_{i>j}\lambda_i$, we instead sum over all possible pairs $(i,j)$, and subtract those for which $i\le j$.
    Summing the term $\lambda_i$ over all possible pairs $(i,j)$ we obtain~$n\ell$.
    Those terms where $i=j$ contribute $\sum_i\lambda_i=n$.
    Thus, we have
    \[\abs{\G^\lambda}=n\ell - n - 2\sum_{i>j}\lambda_j+\sum_i\frac{m_i^2}{2}+\sum_i\frac{m_i}{2}.\]
    From \eqref{eq:n_fn}, we see that $\sum_{i>j}\lambda_j=\sum_j(j-1)\lambda_j=\n(\lambda)$. Furthermore, $\sum_i{m_i}=\ell$, hence
    \[\abs{\G^\lambda}=n\ell - n - 2\n(\lambda)+\sum_i\frac{m_i^2}{2}+\frac{\ell}{2}.\]
\end{proof}

\begin{thm}\label{thm:worst_gap}
    For any $A\in\U_n$ with Jordan form $J_\lambda$, we have
    \[X_AC_{\U}(A)X_A^{-1}\subseteq C(\G^\lambda).\]
\end{thm}

\begin{proof}
    We proceed by induction on $n$.
    In the base case, when $n=1$, we see that~$A$ must be the $1\times1$ zero-matrix, and $X_A$ is the $1\times1$ identity matrix.
    Thus, $X_AC_{\U}(A)X_A^{-1}$ consists of only the zero matrix, which is equal to $C(\G^{(1)})$.

    For $n>1$, let $A\in\U_n$ be $\GL$-conjugate to $J_\lambda$.
    Let $B:=A|_{n-1}$ be $\GL$-conjugate to $J_\mu$.
    The key to our inductive step is the equality 
    \begin{equation}\label{eq:inductive_key}
        C_{\U}(A)=\overline{C_{\U}(B)}\cap C_{\M}(A).
    \end{equation}
    To see this equality, first note that for two upper-triangular matrices $X$ and $Y$ to commute, it must also be that $X|_k$ and $Y|_k$ commute for every $k$.
    Hence, we have~$C_{\U}(A)\subseteq \overline{C_{\U}(B)}$.
    It is immediate that $C_{\U}(A)\subseteq C_{\M}(A)$, proving that the left-hand side is contained in the right-hand side.
    In the other direction, observe that $\overline{C_{\U}(B)}$ consists only of upper-triangular matrices, and that $C_{\M}(A)$ consists only of matrices that commute with $A$.

    We implicitly embed $X_B$ into $\M_{n\times n}$ by identifying $X_B$ with $X_B\oplus1$.
    With this identification, we have 
    \[X_B\overline{V}X_B^{-1}=\overline{X_B V X_B^{-1}}\text{, for any subspace }V\subseteq\M_{(n-1)\times(n-1)}.\]
    Now consider conjugating both sides of \eqref{eq:inductive_key} by $X_B$.
    We obtain
    \[X_B C_{\U}(A)X_B^{-1}\ \subseteq\ \overline{X_B C_\U(B) X_B^{-1}}\ \cap\ (X_B C_\M(A) X_B^{-1}).\]
    By inductive hypothesis, we may assume that $X_B C_\U(B) X_B^{-1}\subseteq C(\G^\mu)$. Thus,
    \[X_B C_{\U}(A) X_B^{-1} \subseteq \overline{C(\G^\mu)}\cap C_{\M}(A'),\]
    where $A'=X_BAX_B^{-1}$.
    Now $A'$ satisfies the hypotheses of Lemma \ref{lem:induct_gap_build}, so we may further conjugate by $Y_{A'}$ to obtain
    \[Y_{A'}X_BC_{\U}(A)X_B^{-1}Y_{A'}^{-1} \subseteq Y_{A'}\left( \overline{C(\G^{\mu})} \cap C_{\M}(A')\right)Y_{A'}^{-1} = C(\psi_r(\G^\mu)).\]
    From Proposition \ref{prop:worst}, we know that $\G^\lambda \le \psi_r(\G^{\mu})$, from which it follows immediately that~$C(\psi_r(\G^{\mu}))\subseteq C(\G^\lambda)$ as desired.
\end{proof}

\subsection{Bounds on $\kk(U_n)$}
We will bound $\kk(U_n)$ by using a map $h:\ell^2(\N)\to\R$.
Specifically, define $h(v):=\|v\|^2-\|v-Lv\|^2$.
The following theorem highlights how we use this function.

\begin{thm}\label{thm:h_bound}
    For every $\lambda\vdash n$, we have \[\abs{\comm(\lambda)}\le \sqrt{n!} q^{\frac{1}{2}(n^2+h(\lambda'))}.\]
\end{thm}

\begin{proof}
    Observe that \[\abs{\comm(\lambda)}=\sum_{A:\shape(A)=\lambda}\abs{C_\U(A)}.\]
    From Theorem \ref{thm:worst_gap} and Proposition \ref{prop:gap_dim}, we know that for every $A$,
    \begin{align*}
        \abs{C_\U(A)} &\le \abs{C(\G^\lambda)}=q^{n\ell-\abs{\G^\lambda}}.
    \end{align*}
    Combining this fact with \eqref{eq:yip}, we see that
    \[\abs{\comm(\lambda)} \le F_\lambda\ q^{n\ell-\abs{\G^\lambda}}\le f^\lambda q^{\binom{n}{2}-\n(\lambda)+n\ell-\abs{\G^\lambda}},\]
    It is a basic result from the representation theory of the symmetric group that $f^\lambda\le \sqrt{n!}$.
    It therefore suffices to show that $\binom{n}{2}-\n(\lambda)+n\ell-\abs{\G^\lambda}\le\tfrac{1}{2}(n^2+h(\lambda'))$, or equivalently, that
    \begin{equation}\label{eq:wanted}
        n\ell-\tfrac{n}{2}-\n(\lambda)-\abs{\G^\lambda}\le \tfrac{1}{2}h(\lambda').
    \end{equation}
    Combining \eqref{eq:n_as_ell_two}, \eqref{eq:m_as_ell_two}, and Proposition \ref{prop:g_bound} with the left-hand side of \eqref{eq:wanted}, we obtain
    \begin{align*}
        n\ell-\frac{n}{2}-\n(\lambda)-\abs{\G^\lambda} &= \frac{n}{2}-\frac{\ell}{2}+\n(\lambda)-\sum_i\frac{m_i^2}{2}\\
        &= -\frac{\ell}{2}+\frac{1}{2}\left( \|\lambda'\|^2-\|\lambda'-L\lambda'\|^2 \right)\\
        &\le \frac{1}{2}h(\lambda'),
    \end{align*}
    which completes the proof.
\end{proof}

\begin{lem}\label{lem:asymptotic_technicalities}
    Let $D\subseteq\ell^1(\N)$ denote the vector space of all $v\in\ell^1(\N)$ satisfying the condition\footnote{The condition that all $v_i$ be non-negative is unnecessary. Because the terms are decreasing, if any term is strictly less than zero, the entire sequence would not be in $\ell^1(\N)$. We include it as part of the condition for the sake of clarity.} that $v_1\ge v_2\ge\dots\ge 0$. Then the following hold:
    \begin{enumerate}
        \item For every $c\in\R$, $h(cv)=c^2h(v)$.
        \item $\der{h}{v_1}\ge\der{h}{v_k}$ for all positive integers $k\ge3$.
        \item For all $v\in D$, we have $h(v) \le 2\|v\|_1v_2-3v_2^2$.
        \item For all $v\in D$ satisfying $v_1\ge\frac{1}{2}\|v\|_1$, and for all $k\ge 4$, we have $\der{h}{v_2}\ge\der{h}{v_k}$.
        \item For $v\in D$ with $v_1\ge \tfrac{1}{2}\|v\|_1$, we have $h(v)\le\|v\|_1v_1-\tfrac{3}{4}v_1^2$.
    \end{enumerate}
\end{lem}

\begin{proof}
    Part 1 follows from the fact that $h$ is a homogenous of degree two on $\ell^2(\N)$.
    For part 2, we compute $\der{h}{v_i}$.
    For $i=1$, we see that $\der{h}{v_1}=2v_2$.
    For $i>1$,
    \[\der{h}{v_i}=2v_{i-1}-2v_i+2v_{i+1}.\]
    Thus, for $k\ge 3$, we have \[\der{h}{v_1}-\der{h}{v_k}=(2v_2-2v_{k-1})+(2v_k-2v_{k+1}).\]
    This is positive, because $v$ is a weakly decreasing sequence.

    For part 3, define $w=(\|v\|_1-v_2,v_2,0,\dots)\in D$.
    Let \[D_k=\{v\in D\mid v_{k+1}=0\}.\]
    Furthermore, define $e_i\in\ell^1(\N)$ to be 1 in the $i$th position, and zero elsewhere.
    First suppose that $v\in D_k$ for some $k\ge 3$.
    From part 2, we know that \[h(v+v_k(e_1-e_k))\ge h(v).\]
    Moreover, $v+v_k(e_1-e_k)\in D_{k-1}$.
    Iterating this process, we see that if $v\in D_k$ for some $k\ge 3$, then we have \[h(v)\le h(\|v\|_1-v_2,v_2,0,\dots)=2\|v\|_1v_2-3v_2^2.\]
    Because $\bigcup_{k=2}^\infty D_k$ is dense in $D$, and both the left- and right-hand sides are continuous, we see that $v\le 2\|v\|_1v_2-3v_2^2$ for every $v\in D$.

    For part 4, note that \[\der{h}{v_2}-\der{h}{v_k}=2[v_1-(v_2-v_3)-(v_{k-1}-v_k)-v_{k+1}].\]
    If $k\ge4$, the quantity subtracted from $v_1$ is less than $\|Lv\|_1=\|v\|_1-v_1$.
    Because~$v_1>\tfrac{1}{2}\|v\|_1$, the quantity $\der{h}{v_2}-\der{h}{v_k}$ must be non-negative.

    Lastly, for part 5, we again consider $D_k=\{v\in D\mid v_{k+1}=0\}$.
    For $v\in D_k$, note that part 4 implies $h(v+v_k(e_2-e_k))\ge h(v)$.
    Because $v_1\ge \tfrac{1}{2}\|v\|_1$, we see $v+v_k(e_2-e_k)\in D_{k-1}$.
    Iterating this process, we eventually reach \[w:=(v_1,\|v\|_1-v_1-v_3,v_3,0,\dots).\]
    Thus, if $v\in D_k$ for some $k\ge 4$, we have
    \[h(v) \le h(w) = -\|v\|_1^2 + 4 \|v\|_1 v_1 - 3 v_1^2 + 4 \|v\|_1 v_3 - 6 v_1 v_3 - 4 v_3^2.\]
    As both the left- and right-hand sides are continuous, and $\bigcup_{k=3}^\infty D_k$ is dense in $D$, this bound holds for all $v\in D$.
    From here it suffices to show that $\|v\|_1 v_1-\tfrac{3}{4}v_1^2-h(w)\ge 0$ for all $v$.
    To this end, note that \[\|v\|_1 v_1-\tfrac{3}{4}v_1^2-h(w)=\tfrac{1}{4}\left( 2\|v\|_1 - 3v_1 - 4v_3 \right)^2\ge 0.\]
    Hence, $h(v)\le h(w)\le \|v\|_1 v_1-\tfrac{3}{4}v_1^2$ as desired.
\end{proof}

\begin{cor}\label{cor:sixth}
    The number of conjugacy classes in the group of upper-triangular matrices is bounded by \[\kk(U_n)\le p(n)\sqrt{n!}\ q^{\frac{n^2}{6}+\frac{n}{2}},\]
    where $p(n)$ is the number of integer partitions of $n$.
\end{cor}

\begin{proof}
    Because \[\abs{\comm(\U_n)}=q^{\binom{n}{2}}\kk(U_n),\] it suffices to show that \[\abs{\comm(\U_n)}\le p(n)\sqrt{n!}\ q^{\frac{2n^2}{3}}.\]
    To this end, we stratify $\comm(\U_n)$ by the Jordan type of the first matrix in each pair.
    From Theorem \ref{thm:h_bound}, we have \[\abs{\comm(\U_n)}=\sum_{\lambda\vdash n}\abs{\comm(\lambda)}\le\sum_{\lambda\vdash n} \sqrt{n!} q^{\frac{1}{2}(n^2+h(\lambda'))}.\]
    The sum is over $p(n)$ terms, so it suffices to show that each term is bounded $\sqrt{n!}q^{\frac{2}{3}n^2}$.
    In other words, it suffices to show that for every $\lambda\vdash n$, we have \[h(\lambda')\le\frac{n^2}{3}.\]

    Let $D=\{v\in\ell^1(\N)\mid v_1\ge v_2\ge\cdots\ge0\}$.
    Let $c\in\R$ such that $v_2=c\|v\|_1$.
    Then from part 3 of Lemma~\ref{lem:asymptotic_technicalities}, we have $h(v)\le (2c-3c^2)\|v\|_1^2$.
    This quantity is maximized at a value of $\tfrac{1}{3}\|v\|_1^2$, by taking $c=\tfrac{1}{3}$.
    For any $\lambda\vdash n$, we therefore have $h(\lambda')\le \tfrac{n^2}{3}$, as desired.
    Thus, our upper bound for the number of pairs of commuting upper-triangular matrices is \[\abs{\comm(\U_n)}\le p(n)\sqrt{n!}\ q^{\frac{2n^2}{3}},\] which implies the desired result.
\end{proof}

We are now ready to prove the following result, which trivially implies Theorem~\ref{thm:conj_bound}.
\begin{thm}\label{thm:better_conj_bound}
    For every positive integer $n$, and every prime power $q$, we have \[\kk(U_n)\le p(n)^2n!\ q^{cn^2+\frac{n}{2}},\text{ where }c=\frac{20\sqrt{2}}{49}-\frac{41}{98}\approx0.15886.\]
\end{thm}

\begin{proof}
    Because $\abs{\comm(\U_n)}=\sum_{\lambda,\mu\vdash n}\abs{\comm(\lambda,\mu)}$,
    it suffices to show that \[\abs{\comm(\lambda,\mu)}\le \sqrt{n!}q^{c n^2+\frac{n}{2}}.\]

    Let $\delta$ and $\eps$ denote small positive quantities, each less than $\tfrac{1}{6}$, to be optimized later.
    We bound $\comm(\lambda,\mu)$ in one of three ways, depending on the shapes of $\lambda$ and $\mu$.
    Let $v\in\ell^1(\N)$ be defined by $v:=\lambda'/n$, so that $\|v\|_1=1$.

    First consider the case where $\abs{v_2-\frac{1}{3}}\ge\delta$.
    From Lemma \ref{lem:asymptotic_technicalities}, we use the bound~$h(v)\le 2v_2-3v_2^2$. 
    As a function of $v_2$, the right-hand side attains its maximum at~$v_2=\tfrac{1}{3}$.
    Thus, for $v_2$ satisfying $\abs{v_2-\frac{1}{3}}\ge\delta$, the right-hand side is maximized at $v_2=\frac{1}{3}\pm\delta$ with a value of $\frac{1}{3}-3\delta^2$.
    Thus, \[\abs{\comm(\lambda,\mu)}\le \sqrt{n!}\;q^{\frac{n^2}{2}(1+h(v))}\le \sqrt{n!}\;q^{\frac{n^2}{2}(\frac{4}{3}-3\delta^2)}.\]
    Because $\abs{\comm(\lambda,\mu)}=\abs{\comm(\mu,\lambda)}$, we also have this bound if $w=\mu'/n$ satisfies~$\abs{w_2-\frac{1}{3}}\ge\delta$.
    Henceforth we may assume that both $v$ and $w$ have their second entry in the open interval $(\tfrac{1}{3}-\delta,\tfrac{1}{3}+\delta)$.

    Second, if $v_1>\frac{2}{3}-\eps$, then from Lemma \ref{lem:asymptotic_technicalities}, we use the bound $h(v)\le v_1-\tfrac{3}{4}v_1^2$.
    Because $\eps\le\tfrac{1}{6}$, we know that $v_1>\tfrac{1}{2}$, and so $v$ satisfies the hypotheses of part 5 of Lemma \ref{lem:asymptotic_technicalities}.
    The polynomial $x-\tfrac{3}{4}x^2$ is maximized at $\tfrac{2}{3}$, and so for $v_1<\frac{2}{3}-\eps$, we see that $h(v)\le \tfrac{1}{3} - \tfrac{3}{4}\eps^2$.
    Thus, \[\abs{\comm(\lambda,\mu)} \le f^\lambda q^{\frac{n^2}{2}(1+h(v))}\le \sqrt{n!}\ q^{\frac{n^2}{2}(\frac{4}{3}-\frac{3}{4}\eps^2)}.\]
    By symmetry, this bound also applies if $\mu'_1\le(\frac{2}{3}-\eps)n$, so we may assume that both~$\lambda'_1/n$ and $\mu'_1/n$ are at least $\tfrac{2}{3}-\eps$.

    Lastly, consider the case where \[\frac{\lambda'_1}{n}\ge \frac{2}{3}-\eps\text{, and }\abs{\frac{\lambda'_2}{n}-\frac{1}{3}}\le \delta,\] and similarly for~$\mu$.
    We obtain an upper bound in the last case by disregarding commutativity, and using the bound \[\abs{\comm(\lambda,\mu)}\le F_\lambda F_\mu=f^\lambda f^\mu q^{n^2-n-\n(\lambda)-\n(\mu)}.\]
    To maximize the exponent, we take $\lambda'=\mu'=\left( (\tfrac{2}{3}-\eps)n,(\tfrac{1}{3}-\delta)n,1,1,\dots,1\right)$.
    The exponent $n^2-n-\n(\lambda)-\n(\mu)$ is therefore bounded above by
    \[n^2\left[1-(\tfrac{2}{3}-\eps)^2-(\tfrac{1}{3}-\delta)^2\right].\]

    Thus, our three bounds on the exponent are given by
    \[n^2\left[\tfrac{2}{3}-\tfrac{3}{2}\delta^2\right],\ n^2\left[\tfrac{2}{3}-\tfrac{3}{8}\eps^2\right]\text{ and},\ n^2\left[1-(\tfrac{2}{3}-\eps)^2-(\tfrac{1}{3}-\delta)^2\right].\]

    Taking $\eps=2\delta=\frac{4}{21}(5-3\sqrt{2})\approx 0.1443$, it is easy to verify that all three of these quantities are equal to $\alpha n^2$, where
    \[\alpha:=\frac{4}{49}+\frac{20\sqrt{2}}{49}\approx 0.65886.\]

    Thus, we obtain the bound $\comm(\U_n)\le \sum_{\lambda,\mu} n! q^{\alpha n^2}$, proving that
    \[\kk(U_n)\le p(n)^2\ n!\ q^{(\alpha-\frac{1}{2})n^2+\frac{n}{2}}=p(n)^2\ n!\ q^{cn^2+\frac{n}{2}}.\]
\end{proof}

\section{Lower central series}\label{sec:lower_central}
The aim of this section is to provide bounds on the number of conjugacy classes in groups in the lower central series of $U_n(q)$.
The techniques used here do not overlap with the techniques of gap arrays introduced in the previous sections.

For a nonnegative integer $k\le n-1$, let $U_{n,k}(q)$ denote the $k$th term in the lower central series of $U_n(q)$.
Explicitly, \[U_{n,k}(q)=\{A\in U_n(q)\mid A_{i,j}=0\text{ whenever }0< i-j\le k\}.\]
In this section, we prove an upper bound on then number of conjugacy classes in $U_{n,k}(q)$.
We first introduce some notation which we will use in the proof.

By $\U_{n,k}(q)$, we mean the $\F_q$-subalgebra of $\U_n(q)$ defined by \[\U_{n,k}(q):=\{X\mid 1+X\in U_{n,k}(q)\}.\]
Note that $\U_n(q)=\U_{n,0}(q)$.
We omit the parameter $q$ when the field is clear from context.
For a finite group $G$, we define the \emph{commuting probability} $\cp(G)$ to be the probability that two elements chosen uniformly at random from $G$ commute.
Clearly \[\cp(G)=\frac{\abs{\comm(G)}}{\abs{G}^2}=\frac{\kk(G)}{\abs{G}}.\]

For $A\in\U_a$ and $B\in\U_b$, let $T_{A,B}:\M_{a\times b}\to \M_{a\times b}$ by $T_{A,B}(X):=AX-XB$.
Note the map $T_{\lambda,\mu}$ defined in Lemma \ref{lem:jordan_rank} is shorthand for $T_{J_\lambda,J_\mu}$.
The following lemma is the key tool used in the proof of Theorem~\ref{thm:lower_central_bound}.

\begin{lem}\label{lem:key_lower_central}
    Let $a$ and $b$ be positive integers, and let $k$ be an integer such that $0\le k<a+b$.
    Let $V=\{X\in \M_{a\times b}\mid X_{i,j}=0\text{ whenever }j-i\ge a-k\}$.
    Then
    \[\cp(a+b,k)=\frac{1}{\abs{\U_{a,k}}^2\cdot\abs{\U_{b,k}}^2}\sum q^{-\dim(\im T_{A_1,B_1}+\im T_{A_2,B_2})}\]
    where the sum is over all $(A_1,A_2)\in\comm(\U_{a,k})$ and all $(B_1,B_2)\in\comm(\U_{b,k})$.
\end{lem}

\begin{proof}
    We begin with the vector space isomorphism $\U_{a+b,k}\to \U_{a,k}\oplus V\oplus\U_{b,k}$ given by
    \[e_{i,j}\mapsto\begin{cases}
            (e_{i,j}, 0, 0) & \text{if }i,j\le a\\
            (0, e_{i,j - a}, 0) & \text{if }i\le a, j>a\\
            (0, 0, e_{i - a, j - a}) & \text{if }i,j>a.
    \end{cases}\]
    Graphically, the isomorphism is shown in Figure \ref{fig:decomposition}.

    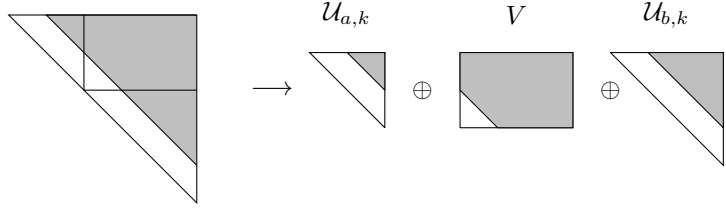
\begin{figure}[h]
        \begin{tikzpicture}[scale=0.5]
            \draw[fill=lightgray] (5,1) -- (1,5) -- (5,5);
            \draw (5,5) -- (5, 0) -- (0,5) -- (5,5);
            \draw (5,3) -- (2,3) -- (2,5);

            \node at (7, 3) {$\longrightarrow$};
            \draw (10, 2) -- (8, 4) -- (10, 4) -- (10, 2);
            \draw[fill=lightgray] (10, 3) -- (9, 4) -- (10, 4) -- (10, 3);
            \node at (11, 3) {$\oplus$};
            \draw (12, 2) -- (12, 4) -- (15, 4) -- (15, 2) -- (12, 2);
            \draw[fill=lightgray] (12, 3) -- (12, 4) -- (15, 4) -- (15, 2) -- (13, 2) -- (12, 3);
            \node at (16, 3) {$\oplus$};
            \draw (16, 4) -- (19, 4) -- (19, 1) -- (16, 4);
            \draw[fill=lightgray] (17, 4) -- (19, 4) -- (19, 2) -- (17, 4);

            \node at (9,5) {$\U_{a,k}$};
            \node at (13.5,5) {$V$};
            \node at (17.5,5) {$\U_{b,k}$};

        \end{tikzpicture}
        \caption{A graphical representation of the decomposition of $\U_{a+b,k}$ into $\U_{a,k}\oplus V\oplus\U_{b,k}$. The gray regions represent the cells which can be non-zero. The white regions represent the cells which must contain zeros.}
        \label{fig:decomposition}
    \end{figure}

    The multiplicative structure on $\U_{a+b,k}$ can be pushed through this isomorphism to make $\U_{a,k}\oplus V\oplus \U_{b,k}$ an $\F_q$-algebra, and is given by
    \[(A_1,X_1,B_1)\cdot (A_2,X_2,B_2)=(A_1B_1,\ A_1X_2+X_1B_2,\ A_2B_2).\]

    We count pairs of commuting elements in $\U_{a,k}\oplus V\oplus \U_{b,k}$.
    Note that $(A_1,X_1,B_1)$ and $(A_2,X_2,B_2)$ commute if and only all three of the following conditions hold:
    \begin{enumerate}
        \item $A_1$ and $A_2$ commute in $\U_{a,k}$,
        \item $B_1$ and $B_2$ commute in $\U_{b,k}$,
        \item $T_{A_1,B_1}(X_2)=T_{A_2,B_2}(X_1)$.
    \end{enumerate}
    For $i=1,2$, let $T_i=T_{A_i, B_i}$.
    We may count pairs of commuting elements in $\U_{a+b,k}$ by counting how many $(X_1,X_2)\in V\oplus V$ satisfy $T_1(X_2)=T_2(X_1)$, and summing over all $(A_1,A_2)\in\comm(\U_{a,k})$ and all $(B_1,B_2)\in\comm(\U_{b,k})$.
    Therefore,
    \[\cp(a+b,k) = \frac{1}{\abs{\U_{a+b,k}}^2}\sum\#\{(X_1,X_2)\in V\times V\mid T_1(X_2)=T_2(X_1)\},\]
    where the sum is over \[S:=\{(A_1,A_2,B_1,B_2)\mid (A_1,A_2)\in\comm(\U_{a,k}),(B_1,B_2)\in\comm(\U_{b,k})\}.\]
    Note that $\{(X_1,X_2)\mid T_1(X_2)=T_2(X_1)\}$ is the kernel of the map $\Phi:V\oplus V\to V$ defined by
    \[\Phi(X,Y):=T_1(X)-T_2(Y).\]
    Obviously, $\im\Phi=\im T_1+\im T_2$.
    By the rank-nullity theorem, we have 
    \begin{equation}\label{eq:nearly_done}
        \cp(a+b,k)=\left(\frac{q^{\dim V}}{\abs{\U_{a+b,k}}}\right)^2\sum_S q^{-\dim(\im T_1 + \im T_2)}.
    \end{equation}

    Recalling that $\dim\U_{a+b,k}=\dim\U_{a,k}+\dim V+\dim\U_{b,k}$, we see that 
    \begin{equation}\label{eq:dimensions}
        \frac{q^{\dim V}}{\abs{\U_{a+b,k}}}=\frac{1}{\abs{\U_{a,k}}\abs{\U_{b,k}}}.
    \end{equation}
    Putting together equations \eqref{eq:nearly_done} and \eqref{eq:dimensions}, we obtain the desired result.
\end{proof}

\begin{lem}\label{lem:no_comm}
    Let $N_r(a,b):=\{(A,B)\in\U_a\times\U_b\mid \rank T_{A,B}\le r\}$.
    Then \[\abs{N_{a,b}(r)}\le p(a)p(b)\sqrt{a!b!}\ q^{\frac{(a-b)^2}{2}+r}.\]
\end{lem}

\begin{proof}
    Because $\dim\ker T_{A,B}=\dim\ker T_{C,D}$ whenever $A$ is $\GL_a$-conjugate to $C$ and $B$ is $\GL_b$-conjugate to $D$, we may safely assume that $A$ and $B$ are in Jordan canonical form.
    From Lemma \ref{lem:jordan_rank}, we know $\dim\ker T_{\lambda,\mu}=\angs{\lambda',\mu'}$.
    Applying the Cauchy-Schwarz inequality followed by the classical arithmetic-geometric mean inequality, we obtain
    \begin{equation}\label{eq:rank_bound}
        \dim\ker T_{\lambda,\mu}=\angs{\lambda',\mu'}\le \frac{\|\lambda'\|^2+\|\mu'\|^2}{2}=\frac{a+b}{2}+\n(\lambda)+\n(\mu).
    \end{equation}

    We stratify $N_{a,b}(r)$ by the Jordan forms of the pairs $(A,B)$ which appear in~$N_{a,b}(r)$.
    Let $S:=\{(\lambda,\mu)\mid \angs{\lambda',\mu'}\ge ab-r\}$.
    The pairs of partitions in $S$ are precisely the partitions indexing Jordan forms of pairs of matrices in $N_{a,b}(r)$.
    Thus, \[\abs{N_{a,b}(r)}=\sum_{(\lambda,\mu)\in S} F_\lambda(q)F_\mu(q).\]
    Recall from \eqref{eq:yip} that $F_\lambda(q)\le f^\lambda q^{\binom{a}{2}-\n(\lambda)}$, so
    \[\abs{N_{a,b}(r)} \le \sum_{(\lambda,\mu)\in S} f^\lambda f^\mu q^{\frac{a^2+b^2}{2} - \left( \frac{a+b}{2} + \n(\lambda) + \n(\mu) \right)}.\]
    Now applying \eqref{eq:rank_bound} we get
    \[\abs{N_{a,b}(r)}\le \sum_{(\lambda,\mu)\in S} f^\lambda f^\mu q^{\frac{a^2+b^2}{2} - (ab - r)} = q^{(a-b)^2/2+r}\ \sum_{(\lambda,\mu)\in S} f^\lambda f^\mu.\]

    Noting that $f^\lambda\le \sqrt{a!}$ and $\abs{S}\le p(a)p(b)$, the result follows.
\end{proof}

We now present the proof of Theorem \ref{thm:lower_central_bound}.
Define $\beta_0:=0$, and for each positive integer $m$, define $\beta_m$ inductively by \[\beta_m:=\tfrac{1}{4}\left(\beta_{m-1}-(1-2^{-m})^2\right).\]
A routine calculation shows that for all $m$, we have \[\beta_m=-\frac{1}{3}-\frac{2}{3}\cdot 4^{-m}+2^{-m}- 4^{-(m+1)}m.\]

\begin{thm}\label{thm:lcs_comm}
    Let $m$ be a non-negative integer such that $2^{-(1+m)}\le \tfrac{k}{n}\le 2^{-m}$. Then \[\cp(n,k)\le q^{\beta_m n^2(1+o_m(1))},\] where $o_m(1)$ is a function of $m$ and $n$ which, for each fixed $m$, tends to zero as $n$ tends to infinity.
\end{thm}

\begin{proof}
    We proceed with the proof that $\cp(n,k)\le q^{\beta_m n^2(1+o_m(1))}$ by induction on $m$.
    When $m=0$, we have $\tfrac{1}{2}\le \tfrac{k}{n}\le 1$, and so $U_{n,k}$ is abelian, proving that $\cp(n,k)=1=q^0$ as desired.

    Now let $V=\{X\in \M_{a\times b}\mid X_{i,j}=0\text{ whenever }j-i\ge a-k\}$, and let
    \begin{align*}
        C_{a,b}^k(r) &= \{(A_1,A_2,B_1,B_2)\in\comm(\U_{a,k})\times\comm(\U_{b,k})\mid\\
        & \hspace{5cm}\rank_V(T_{A_1,B_1},T_{A_2,B_2})=r\}.
    \end{align*}
    From Lemma \ref{lem:key_lower_central}, we have 
    \[\cp(n,k)=\frac{1}{\abs{\U_{a,k}}^2\abs{\U_{b,k}}^2}\sum_{j=0}^{ab-\binom{k+1}{2}}q^{-j}\ \abs{C_{a,b}^k(j)}.\]
    Pick some number $r$ to be optimized later, and split the sum based on the relationship between $r$ and $j$.
    First considering the case when $j\ge r$, we have
    \begin{equation}\label{eq:large_rank}
        \sum_{j\ge r}q^{-j}\abs{C_{a,b}^k(j)} \le q^{-r}\sum_{j\ge r}\abs{C_{a,b}^k(j)} \le q^{-r}\abs{\comm(\U_{a,k})}\abs{\comm(\U_{b,k})}.
    \end{equation}
    
    On the other hand, if $j < r$, we forget about the commutativity relation, and remember only that the rank of each map $T_{A_i,B_i}$ must be bounded by $r$.
    Thus, for $j<r$ we have $C_{a,b}^j(k)\le \abs{N_{a,b}(j)}^2$.
    Combining this with \eqref{eq:large_rank} and Lemma \ref{lem:no_comm}, we obtain
    \begin{align*}
        \cp(n,k) &\le q^{-r}\cp(a,k)\cp(b,k)+\frac{1}{\abs{\U_{a,k}}^2\abs{\U_{b,k}}^2}\sum_{j=0}^{r-1}q^{-r}\abs{N_{a,b}(j)}^2\\
        &\le q^{-r}\cp(a,k)\cp(b,k)+\frac{p(a)^2p(b)^2a!b!}{\abs{\U_{a,k}}^2\abs{\U_{b,k}}^2}\sum_{j=0}^{r-1}q^{(a-b)^2+2j-r}\\
        &\le q^{-r}\cp(a,k)\cp(b,k)+\frac{p(a)^2p(b)^2a!b!}{\abs{\U_{a,k}}^2\abs{\U_{b,k}}^2}q^{(a-b)^2+r-1}.
    \end{align*}

    To optimize these parameters, we take $a=\floor{n/2}$ and $b=\ceil{n/2}$, and 
    \begin{equation}\label{eq:optimize_r}
        r=\left[\frac{1}{4}\left(\beta_{m-1}+\left(1-2^{-m}\right)^2\right)n^2\right],
    \end{equation}
    where the square brackets denote the nearest integer function.
    In the calculations to follow, we omit the $[\cdot]$, as it complicates the computation, and does not contribute to the leading term in the exponent.
    The difference is absorbed into the $o_m(1)$ term.

    Because $2^{-m}\le \tfrac{2k}{n}\le 2^{1-m}$, by inductive hypothesis, we have \[\cp(a,k)\le q^{\beta_{m-1}a^2(1+o_m(1))}=q^{\frac{\beta_{m-1}n^2}{4}(1+o_m(1))}.\]
    Note that $p(a)^2p(b)^2a!b!\le q^{o(n^2)}$, and that 
    \[\abs{\U_{a,k}}=q^{\frac{(a-k)^2}{2}(1+o_m(1))}=q^{\frac{1}{8}\left(1-\frac{2k}{n}\right)^2n^2(1+o_m(1))}.\]
    Combining these facts, we obtain
    \[\cp(n,k) \le q^{-r+\frac{\beta_{m-1}n^2}{2}(1+o_m(1))}+q^{r-\frac{1}{2}(1-2^{-m})^2n^2(1+o_m(1))}\]
    Substitiuting in \eqref{eq:optimize_r}, both exponents on the right-hand side become 
    \[\tfrac{1}{4}\left(\beta_{m-1}-(1-2^{-m})^2\right)n^2(1+o_m(1)),\]
    which is equal to $\beta_mn^2(1+o_m(1))$.
    Thus, \[\cp(n,k)\le 2q^{\beta_mn^2(1+o_m(1))}=q^{\beta_mn^2(1+o_m(1))},\] which completes the proof.
\end{proof}

\begin{cor}[Theorem \ref{thm:lower_central_bound}]
    Let 
    \[\gamma_m:=\frac{1}{6}-\frac{13}{24}\cdot 4^{-m}+2^{-(m+1)}-4^{-(m+1)}m.\]
    Then for every $q$,
    \[\kk(U_{n,k}(q))\le q^{\gamma_m n^2(1+o_m(1))},\]
    where $m=\floor{\log_2\left(\tfrac{n}{k}\right)}$, and $o_m(1)$ denotes a function which, for each fixed $m$, tends to zero as $n$ tends to infinity.
\end{cor}

\begin{proof}
    Recalling that $\kk(G)=\cp(G)\abs{G}$, we see that Theroem~\ref{thm:lcs_comm} implies that \[\kk(U_{n,k})\le q^{\binom{n-k}{2}+\beta_mn^2(1+o_m(1))}.\]
    Combining this with the fact that $\binom{n-k}{2}=\tfrac{1}{2}\left(1-\tfrac{k}{n}\right)^2n^2(1+o_m(1))$, and also the fact that $\frac{k}{n}\ge 2^{-(1+m)}$,
    we obtain
    \[\kk(G) \le q^{\left(\tfrac{1}{2}\left( 1-2^{-(1+m)} \right)^2+\beta_m\right)n^2(1+o_m(1))} \le q^{\gamma_mn^2(1+o_m(1))},\]
    where $\gamma_m=\frac{1}{6}-\frac{13}{24}\cdot 4^{-m}+2^{-(m+1)}-4^{-(m+1)}m$.
\end{proof}

\section{Final Remarks}
\subsection{}
The original motivation for determining asymptotic bounds on $\kk(U_n)$ was that such bounds could be used to give upper bounds on the number of $p$-groups of a given order.
Specifically, Higman used his upper bound of $\kk(U_n(p))\le p^{\frac{n^2}{4}}$ to show that there were no more than $p^{\frac{2n^3}{15}(1+o(1))}$ groups of order $p^n$ \cite{higman1960enum1}.
In the same paper, he provided the lower bound of $p^{\frac{2}{27}n^3(1+o(1))}$ for the number of groups of order $p^n$, by explicitly exhibiting so many $p$-groups of class 2, but was unable to close the gap between the constants $2/27$ and $2/15$.
He noted that a reduction in the upper bound for $\kk(U_n(q))$ would result in a reduction for the upper bound on the number of $p$-groups, but that even proving $\kk(U_n(q))\le q^{\frac{n^2}{12}(1+o(1))}$ could not completely close the gap with his methods.

In 1965, Charles Sims used a different argument to show that the number of groups of order $p^n$ is $p^{\frac{2n^3}{27}+O(n^{8/3})}$, resolving Higman's original question~\cite{sims1965enumerating} (see also, \cite{blackburn2007enumeration}).

\subsection{}
Several other people have studied conjugacy in $U_n$ since Higman first introduced the question.
Gudivok et al.~provide an algorithm to compute the number of conjugacy classes in $U_n$ for arbitrary $q$, and computed the number for $n\le 7$~\cite{gudivok1990classes}.\footnote{The paper actually computed the number of conjugacy classes through $n=9$, though the implementation had an error which yielded the wrong values for $n=8$ and $9$.
The polynomials provided do evaluate to the correct values for $\kk(U_n(2))$.}
Vera-L\'opez and Arregi constructed a different algorithm and were able to compute $\kk(U_n)$ for arbitrary $q$ and $n\le 13$ \cite{vla2003conjugacy}.
Moreover, they verified that for $n\le13$, the quantity $\kk(U_n)$ was a polynomial in $q$ with integer coefficients of degree~$[(n^2+6n)/12]$, where $[\cdot]$ denotes the nearest integer function.
In \cite{paksoffer2015}, we verify that $\kk(U_n)$ is a polynomial in $q$ with integer coefficients of the same degree for~$n\le16$.
While there is evidence that $\kk(U_n)$ may not be a polynomial in $q$ for large~$n$ (see \cite{halasi2011number,paksoffer2015}), we believe that Higman's lower bound is asymptotically tight: 
\begin{conj}
    For every prime power $q$, \[\kk(U_n)=q^{\frac{n^2}{12}(1+o(1))}.\]
\end{conj}

\subsection{} The proof of Corollary \ref{cor:sixth} shows that our bound is maximized for partitions with at most two columns.
However, a straightforward inductive argument proves that if $A$ has Jordan form $J_\lambda$, and $\lambda$ has at most two columns, then there exists a gap array $G$ such that $X_AC_\U(A)X_A^{-1}=C(G)$.
Moreover, we can explicitly compute the size of this gap array, and determine that $\comm(\lambda)=q^{o(n^2)}$.

\subsection{} Theorem \ref{thm:lower_central_bound} provides another proof that $\kk(U_n)\le q^{\frac{n^2}{6}(1+o(1))}$.
Taking $m$ to infinity, we see that $\gamma_m$ approaches $\tfrac{1}{6}$.
To make this precise, one must also show that the the term $o_m(1)$ does not grow too fast with $m$.
In \cite{soffer2016thesis}, we do this explicitly.

\subsection{} A \emph{pattern group} is a subgroup of $U_n$ defined by prescribing cells which are allowed to be non-zero (see \cite{isaacs2007counting}).
The normal pattern groups in $U_n$ were characterized by Weir, and have the property that if a cell is allowed to be non-zero, then so is any cell to the right of it or above it \cite{weir1955sylow}.
Due to this combinatorial structure we believe that similar techniques can be used to bound the number of conjugacy classes.

\subsection{}
Regarding the bounds on $\kk(U_{n,k})$, Lemma \ref{lem:no_comm} has room for improvement.
Specifically, when we wish to bound the number of matrices in $U_{n,k}$ which have Jordan type $\lambda$, we use $F_\lambda$, the number of such matrices in $U_{n}$.
If we had estimates on the number of matrices in $U_{n,k}$ of a particular Jordan type, we could tighten the upper bound on~$\kk(U_{n,k})$.
We hope to return to this problem in the future.

\section{Acknowledgements}
I would like to thank my Ph.D. advisor, Igor Pak, for introducing me to this problem, and for his help with early revisions of this paper.
I would like to specially thank Persi Diaconis and Eric Marberg for pointing us to \cite{marberg2008constructing}, and Antonio Vera-L\'opez for directing us to \cite{vla1992conjugacy}.
Their helpful comments and discussions were invaluable.
I am also grateful to Adam Azzam and Alejandro Morales for their many helpful remarks, both mathematical and expository.

\end{document}